\documentclass[12pt,reqno,oneside]{amsart}
\usepackage[a4paper,marginparwidth=2cm]{geometry}
\usepackage[utf8]{inputenc}
\usepackage[T1]{fontenc}
\usepackage{amsthm}
\usepackage{amsmath}
\usepackage{amsfonts}
\usepackage{mathtools}
\usepackage{enumitem}
\usepackage{marginnote}
\usepackage{xcolor}
\usepackage{hyperref}
\usepackage{amssymb}

\setlist[enumerate]{label=(\alph*)}
\numberwithin{equation}{section}

\definecolor{Darkgreen}{rgb}{0,0.4,0}
\hypersetup{
  colorlinks,
  linkcolor=Darkgreen,
  citecolor=Darkgreen,
  breaklinks=true,
  bookmarksnumbered=true,
  plainpages=false,
  unicode=true,
  pdfauthor={Jiří Černý, Ramon Locher},
  pdftitle={Critical and near-critical level-set percolation of the Gaussian free field on regular trees}
  }

\newtheorem{theorem}{Theorem}[section]
\newtheorem{proposition}[theorem]{Proposition}
\newtheorem{lemma}[theorem]{Lemma}
\newtheorem{corollary}[theorem]{Corollary}

\theoremstyle{definition}

\theoremstyle{remark}
\newtheorem{remark}[theorem]{Remark}

\newcommand{\R}{\mathbb{R}}

\newcommand{\N}{\mathcal{N}}

\newcommand{\T}{\mathbb{T}}
\newcommand{\Id}{\mathrm{Id}}

\newcommand{\defeq}{\coloneqq}
\newcommand{\eqdef}{\eqqcolon}

\newcommand{\D}{\mathop{}\!\mathrm{d}}

\let\phi\varphi

\DeclareMathOperator{\spn}{span}
\DeclareMathOperator{\codim}{codim}

\DeclarePairedDelimiter{\abs}{\lvert}{\rvert}

\DeclarePairedDelimiter{\norm}{\lVert}{\rVert}
\DeclarePairedDelimiter{\set}{\lbrace}{\rbrace}
\DeclarePairedDelimiterX{\inp}[2]{\langle}{\rangle}{#1, #2}

\begin{document}

\title[Critical percolation of the GFF on regular trees]
{Critical and near-critical level-set percolation of the Gaussian free field on regular trees}

\author{Jiří Černý, Ramon Locher}

\begin{abstract}
  For the Gaussian free field on a $(d+1)$-regular tree with $d \geq 2$, we
  study the percolative properties of its level sets in the critical
  and the near-critical regime. In particular, we show the continuity of
  the percolation probability, derive an exact
  asymptotic tail estimate for the cardinality of the connected component
  of the critical level set, and describe the asymptotic behaviour of
  the percolation probability in the near-critical regime.
\end{abstract}

\maketitle

\section{Introduction}
\label{sec:introduction}
In this paper we study the level-set percolation for the discrete
Gaussian free field on regular trees, with focus on its properties in the
critical and the near-critical regime. Our results include the continuity
of the percolation function, an exact asymptotic
tail estimate for the cardinality of the connected component of the critical
level set, and describe the asymptotic behaviour of the percolation
probability in the near-critical case.

The level-set percolation of the Gaussian free field, in particular on $\mathbb
Z^d$, is one of the most important and studied percolation models with long
range dependencies, with first studies dating back to 1980s,
\cite{MolSte1983,LebSal86,BriLebMae87}. In the past decade, a new wave of
results on this model was initiated by \cite{RodSzn13}, where it was shown that,
on $\mathbb Z^d$, this model exhibits a non-trivial percolation phase transition
at a critical level $h^* = h^*(d)$ in any dimension $d\ge 3$. In the subsequent
papers, see for instance \cite{RodSzn13, DreRatSap14, PopRat15, DrePreRod18,
Szn19.2, ChiNit20, GosRodSev22, PanSev22}, the sub- and supercritical phases of
the model were
understood thoroughly, often making use of additional natural critical points in
order to work in a strongly sub-/super-critical regime. In the remarkable paper
\cite{DGRS20}, it was then shown that all those critical points agree with
$h^*$, that is the percolation phase transition is sharp (for a recent,
simpler, and more general proof of the sharpness see \cite{Mui22}).

Compared to the sub- and supercritical regime, the critical and
near-critical regimes are much less understood. On $\mathbb Z^d$, the
situation is to some extend similar to the Bernoulli percolation: it is
not known whether the percolation probability  is continuous at $h^*$,
and the existence of various critical exponents is only conjectured. Only
very recently, \cite{Mui22} provided a first (conjecturally not-optimal)
upper bound on the critical exponent $\beta $ involved in near-critical
asymptotics of the percolation probability.

Incidentally, the critical behaviour is much better understood on the related
model of Gaussian free field on the metric graph of $\mathbb Z^d$, where the
continuity of the percolation function is known \cite{DinWir20,
DrePreRod22}, and various critical exponents were computed in
\cite{DrePreRod21}.

Here, we study the critical behaviour in a considerably simpler
situation, for the level-set percolation of the Gaussian free field on
regular trees. This model was initially investigated in \cite{Szn15}
where the critical value $h^*$ was characterised as the largest
eigenvalue of certain integral operator, and a coupling with random
interlacements was used to derive bounds on  $h^*$, implying in
particular that $0 < h^* < \infty$. Later, in \cite{AbaCer19}, the  sub-
and supercritical phase of the model was studied in detail. Their results
include the continuity of the percolation probability away from the
critical level $h^*$, and rather precise estimates for the cardinality of
the connected components of the level sets in the sub- and super-critical
phase. We complement these results with critical and near-critical
estimates.

Similarly to \cite{Szn15} and, in particular, to \cite{AbaCer19}, we will
take advantage of a connection of the Gaussian free field on regular
trees to certain multi-type branching processes (cf.
Section~\ref{sec:notation_and_results} below). The analysis of these
branching processes is not completely straightforward, as their type
space is uncountable and unbounded and they do not satisfy the conditions
used in the classical literature on branching processes \cite{Harris63, Mode71, AthNey72}.
Fortunately, these conditions can be substituted by certain
hypercontractivity estimates
(cf.~Proposition~\ref{prop:L_h-hypercontractivity} below), which have already been
featured in the previous works.  We are also not aware of any results
about near-critical multi-type branching process which resemble our
analysis of the near-critical behaviour of the percolation probability.

\section{Model and results}

We now define our model. Let $\T$ be the infinite $(d+1)$-regular tree,
with $d\ge 2$, rooted at an arbitrary fixed vertex $o\in \T$. On $\T$ we consider
the Gaussian free field $\phi = (\phi_x)_{x\in \T}$, which is a centred
Gaussian process whose covariance function agrees with the Green function
of the simple random walk on $\T$, see \eqref{def:green-func} below for
the precise definition. We use $P$ to denote the law of this process on
$\R^\T$, and, for $a \in \R$, we write $P_a$ for the conditional
distribution of $\phi$ given that $\phi_o=a$,
\begin{equation}
  \label{def:conditional-prob}
  P_a[\, \cdot\, ] \defeq P[\, \cdot\mid \phi_o = a].
\end{equation}
(For an explicit construction of $P_a$ see \eqref{eq:BP-representation}
and the paragraph below it.) Let further $\bar{o} \in \T$ be an arbitrary
fixed neighbour of the root $o$ and define the forward tree $\T^+$ by
\begin{equation}
  \label{def:Tplus}
  \T^+ \defeq \set{x \in \T :
    \bar{o} \text{ is not contained in the geodesic path between $o$ and $x$}}.
\end{equation}

We analyse the percolation properties of the (super-)level sets of $\phi$
above level $h\in \mathbb R$, that is of
\begin{equation}
  E_{\phi}^h \defeq \set{x \in \T : \phi_x \geq h}.
\end{equation}
In particular, we are interested in the connected component of this set
containing the root $o$,
\begin{equation}
  \label{def:Cxh}
  \mathcal{C}_o^h \defeq \set{y \in \T : y \text{ is connected to } o \text{ in } E_\phi^h},\quad h\in\R.
\end{equation}
The critical height $h^*$ of the level-set percolation is defined by
\begin{equation}
  \label{def:hstar}
  h^* = h^*(d) \defeq \inf\set{h \in \R : P[\abs{\mathcal{C}_o^h} = \infty] = 0}.
\end{equation}

It is well known that $h^*$ is non-trivial, more precisely
$h^*\in (0,\infty)$, see \cite[Corollary 4.5]{Szn15}. Moreover, as proved
in \cite{Szn15}, $h^*$ can be characterised with help of the operator
norms of a certain family of non-negative operators $(L_h)_{h\in \R}$
acting on the space $L^2(\nu )$, where $\nu $ is the centred Gaussian
measure with variance $\sigma^2_\nu = d/(d-1)$. We give more details of this
characterisation in Section~\ref{sec:notation_and_results} below. Here,
we only define $\lambda_h$ to be the largest eigenvalue of $L_h$ and
$\chi_h$ the corresponding normed eigenfunction, and recall that $h^*$ is
the unique solution to
\begin{equation}
  \lambda_{h^*} = 1.
\end{equation}
Since we will mostly deal with the
critical case, we often abbreviate
\begin{equation}
  \chi \defeq \chi_{h^*} \quad \text{and} \quad L \defeq L_{h^*}.
\end{equation}

For $a,h\in \mathbb R$ we further introduce
\emph{conditioned} percolation and forward percolation probabilities by
\begin{equation}
  \label{def:percolation-prob}
  \eta(h, a) \defeq P_a[\abs{\mathcal{C}_o^h} = \infty] \quad \text{and}
  \quad \eta^+(h, a) \defeq P_a[\abs{\mathcal{C}_o^h \cap \T^+} = \infty].
\end{equation}
It is known that both of these functions are identically $0$ when $h>h^*$, and
for $h<h^*$ they are strictly positive iff $a\in [h,\infty)$, see
\cite[Proposition 3.3 and its proof]{Szn15}.

Our first two results consider the behaviour of $\mathcal C_0^{h}$ at the
critical height $h=h^*$. The first interim result shows that there is no
percolation at $h^*$.

\begin{theorem}
  \label{thm:crit_BP_dies}
  For all $a\in \R$,
  \begin{equation}
    \eta(h^*, a) = \eta^+(h^*,a) = 0,
  \end{equation}
  and, as consequence,
  \begin{equation}
    P[\abs{\mathcal{C}_o^{h^*} \cap \T^+} = \infty] = P[\abs{\mathcal{C}_o^{h^*}} = \infty] =
    0.
  \end{equation}
\end{theorem}

As corollary of this theorem, we directly obtain the continuity of the
percolation functions. This extends
Theorem~5.1~of \cite{AbaCer19}, where it is shown that the functions
$h\mapsto \eta (h,a)$ and $h\mapsto \eta^*(h,a)$ are left-continuous
everywhere and continuous on $\mathbb R\setminus \{h^*\}$.

\begin{corollary}
  \label{cor:continuity}
  The functions $h\mapsto \eta (h,a)$ and $h\mapsto \eta^*(h,a)$ are
  continuous for every $a\in \mathbb R$.
\end{corollary}

The second result considers the cardinality of $\mathcal{C}_o^{h^*}$ in
the critical case, and describes its exact asymptotic tail behaviour. In
particular it gives a probabilistic meaning to the eigenfunction $\chi$.

\begin{theorem}
  \label{thm:cluster_size_tail}
  For every $a \geq h^*$, as $t\to\infty$,
  \begin{align}
    \label{eq:thm_cst_1}
    P_a[\abs{\mathcal{C}_o^{h^*} \cap \T^+} > t]
    &= C_1 \chi(a) t^{-1/2} \big(1+o(1)\big),\\
    \label{eq:thm_cst_2}
    P_a[\abs{\mathcal{C}_o^{h^*}} > t]
    &=  C_1 d^{-1}(d+1)\chi(a) t^{-1/2} \big(1+o(1)\big),
  \end{align}
  where, denoting by $\inp\cdot\cdot_\nu $ the scalar product on
  $L^2(\nu )$, the constant $C_1$ is given by
  \begin{equation}
    \label{eq:Cx}
    C_1 = \frac{1}{\Gamma(1/2)}\sqrt{\frac{2d}{d-1}\cdot
      \frac{\inp{1}{\chi}_{\nu}}{\inp{\chi}{\chi^2}_{\nu}}}.
  \end{equation}
\end{theorem}

\begin{remark}
  Theorem~\ref{thm:crit_BP_dies} could be seen as a corollary to
  Theorem~\ref{thm:cluster_size_tail}, but we prefer to state it
  separately, since the former theorem is used in the
  proof of the latter one.
\end{remark}

\begin{remark}
  Combining Theorem~\ref{thm:cluster_size_tail} with the stochastic
  domination (see \eqref{eq:stoch_dom} below), it follows that
  \begin{equation}
    a\mapsto \chi_{h^*}(a) \text{ is non-decreasing},
  \end{equation}
  which, to our knowledge, was not known previously.
\end{remark}

Our third result considers the percolation probabilities in the
near-critical supercritical regime. We are able to describe their
asymptotic behaviour as $h \uparrow h^*$. As in Theorem
\ref{thm:cluster_size_tail}, the limiting objects can be expressed in
terms of the eigenfunction~$\chi$.
\begin{theorem}
  \label{thm:percolation_prob}
  The percolation probabilities $\eta $ and $\eta^+$ can be written as
  \begin{align}
    \label{eq:thm_percolation_prob_1}
    \eta^+(h,a) &= C_2 (h^* - h) \big(\chi(a) + r^+_h(a)\big),\\
    \label{eq:thm_percolation_prob_2}
    \eta(h, a) &= C_2\, \frac{d+1}d  (h^* - h) \big(\chi(a) + r_h(a)\big),
  \end{align}
  where for an arbitrary $\varepsilon \in (0,1]$ the reminder functions $r_h$ and $r^+_h$ satisfy
  \begin{equation}
    \lim_{h\uparrow h^*} \norm {r_h}_{L^{2-\varepsilon }(\nu )}
    = \lim_{h\uparrow h^*} \norm {r^+_h}_{L^{2-\varepsilon }(\nu )} = 0,
  \end{equation}
  and where, with $\Id$ denoting the identity on $\mathbb R$, the constant $C_2$
  is given by
  \begin{equation}
    C_2 = 2\cdot \frac{d-1}{d+1} \frac{\inp{\Id}{\chi^2}_\nu}{\inp{\chi}{\chi^2}_\nu}.
  \end{equation}
\end{theorem}

\begin{remark}
  As a consequence of Theorem~\ref{thm:percolation_prob}, the critical exponent
  $\beta $ defined by
  $\beta = \lim_{h\uparrow h^*} \frac{\log P(\abs{\mathcal C_0^h} =
      \infty)}{\log(h^* - h)}$ satisfies $\beta =1$. This coincides with
  the conjectured value of this exponent for the Gaussian free field on
  $\mathbb Z^d$, as well as with its rigorously proved value
  on a large family of metric graphs, \cite[Corollary 1.2]{DrePreRod21}.
\end{remark}

We now briefly discuss the structure of this article. In
Section~\ref{sec:notation_and_results} we introduce more notation and collect useful
known facts about the Gaussian free field on $\T$.
The proof of Theorem~\ref{thm:crit_BP_dies}, which used the techniques known from the
theory of multi-type branching processes, and exploits the convergence of a certain
non-negative martingale (see \eqref{def:M_n}), is given in
Section~\ref{sec:proof_theorem1}.
In Section~\ref{sec:proof_theorem2} we give the proof of
Theorem~\ref{thm:cluster_size_tail}, using the Tauberian theory and investigating
the Laplace transform of the cardinality of $\mathcal{C}_o^{h^*}$. Finally, in
Section~\ref{sec:proof_theorem3}, we use results on bifurcations on Banach
spaces together with the defining equation for the forward percolation probability
introduced in \cite{AbaCer19} (see \eqref{eq:percolation-prob}) to
prove Theorem~\ref{thm:percolation_prob}.

\section{Notation and useful results}
\label{sec:notation_and_results}

In this section we introduce the notation used throughout the paper and recall
several known facts concerning the level set percolation of the
Gaussian free field on trees.

As already stated in the introduction, we use $\T$ to denote the $(d+1)$-regular
tree, $d\ge 2$,  that is an infinite tree whose every vertex has exactly
$d+1$ neighbours. For two vertices $x,y \in \T$ we use $d(x, y)$ to
denote their usual graph distance. The tree is rooted at a fixed arbitrary vertex $o \in \T$,
and $\bar{o} \in \T$ denotes a fixed neighbour of $o$. $\T^+$ denotes the
forward tree as defined in \eqref{def:Tplus}.

We consider the Gaussian free field $\phi = (\phi_x)_{x\in \mathbb T}$ which is the
centred Gaussian process  on $\mathbb T$ whose covariance function is the Green function
of the simple random walk on $\T$, that is
\begin{equation}
  \label{def:green-func}
  E[\phi_x \phi_y] = g(x, y)
  \defeq \frac{1}{d+1} \mathbb{E}_x\Big[\sum_{k=0}^\infty 1_{X_k =
      y}\Big],
  \qquad x,y \in \T,
\end{equation}
where $\mathbb E_x$ stands for the expectation with respect to the simple
random walk $(X_k)_{k \geq 0}$ on $\mathbb T$ starting at $x\in \T$.

We frequently use the fact that  that the Gaussian free
field on $\T$ can be viewed as multi-type branching process with a continuous
type space (see \cite[Section~3]{Szn15} and \cite[Section 2.1]{AbaCer19}).
To this end, we define
\begin{equation}
  \label{eq:sigmas}
  \sigma_\nu^2 \defeq \frac{d}{d-1}
  \quad \text{and} \quad
  \sigma_Y^2 \defeq \frac{d+1}{d},
\end{equation}
and let $(Y_x)_{x\in\T}$ be a collection of independent centred
Gaussian random variables on some probability space
$(\Omega, \mathcal A, P')$ such that $Y_o \sim \N(0, \sigma_\nu^2)$ and
$Y_x \sim \N(0, \sigma_Y^2)$ for $x \neq o$. We then recursively define
another field $\widetilde\phi $ on $\mathbb T$ by
\begin{equation}
  \label{eq:BP-representation}
  \begin{split}
    \text{(a)}\ &\widetilde{\phi}_o \defeq Y_o, \\
    \text{(b)}\ &\parbox[t]{12.5cm}{
      for $x\neq o$,  $\widetilde{\phi}_x \defeq \frac{1}{d} \widetilde{\phi}_{\bar{x}} +
      Y_x$ where $\bar x$ is the direct ancestor of $x$ in $\mathbb T$,
      that is the first vertex on the geodesic path from $x$ to $o$
      different from $x$.}
  \end{split}
\end{equation}
As explained e.g.~in \cite[(2.9)]{AbaCer19}, the law of
$(\widetilde{\phi}_x)_{x\in\T}$ under $P'$ agrees with the law $P$ of the
Gaussian free field $\phi$. Therefore, we will always assume that the
considered Gaussian free field is constructed in this way and will not
distinguish between $\phi$ and $\widetilde \phi$.

Representation \eqref{eq:BP-representation} of $\phi$ can be used to give
a concrete construction for the conditional probability $P_a$ introduced
in \eqref{def:conditional-prob}: It is sufficient to replace (a) in
\eqref{eq:BP-representation} by $\widetilde \phi_0= a$. In addition,
\eqref{eq:BP-representation} easily allows to construct a monotone
coupling of $P_a$ and $P_b$. As the result we obtain:
\begin{equation}
  \label{eq:stoch_dom}
  \text{If $a<b$, then $P_b$ stochastically dominates $P_a$,}
\end{equation}
that is $E_a[f(\phi)] \le E_b[f(\phi)]$ for every bounded increasing
function $f:\mathbb R^{\mathbb T}\to \mathbb R$.

From the construction \eqref{eq:BP-representation} it follows that the
root $o$ can be viewed as an initial
particle of a multi-type branching process; its type is distributed as
$Y_0$. Every particle $\bar x$ in this branching process  has then $d$ offsprings ($d+1$
  if $\bar x = 0$) whose types are independently given by
$\frac{1}{d}{\phi}_{\bar x} + Y$, with $Y\sim N(0,\sigma_Y^2)$.

The branching process point of view can be adapted to $\mathcal{C}_o^h$, by
considering the same multi-type branching process but killing all particles with
type lower than $h$ (and thus also not allowing them to have descendants
  themselves).  Similarly, $\mathcal{C}_o^h \cap \T^+$ can be constructed the same
way, with the only difference that in this case also the root node $o$ has $d$
potential descendants, instead of $d+1$.
We denote by $Z_n^h$ the $n$-th generation of this
branching process
\begin{equation}
  \label{eq:Z_n^h}
  Z_n^h \defeq \set{x \in  \mathcal C_o^h \cap \T^+ : d(o,x) = n },
  \qquad h\in \mathbb R, n\in \mathbb N.
\end{equation}

We now recall more in detail the spectral machinery introduced in
\cite{Szn15} in order to characterise the critical value $h^*$.  Let $\nu$
be a centred Gaussian measure on $\mathbb R$ with variance $\sigma_\nu^2$
(as defined in \eqref{eq:sigmas}), and let $Y$ be a centred Gaussian
random variable with variance $\sigma_Y^2$. The expectation with respect
to this random variable is denoted $E_Y$. We consider the Hilbert space
$L^2(\nu) \defeq L^2(\R, \mathcal B(\R), \nu)$, and for every $h \in \R$
we define the  operator $L_h$ on $L^2(\nu)$ by
\begin{equation}
  \label{eq:L_h}
  \begin{split}
    L_h[f](a)
    &\defeq 1_{[h,\infty)}(a)\, d\, E_Y\Big[1_{[h,\infty)}
      \Big(Y + \frac{a}{d}\Big)f\Big(Y+\frac{a}{d}\Big)\Big]\\
    &= 1_{[h,\infty)}(a)\, d \int_{[h, \infty)} f(x)
    \rho_Y\Big(x - \frac{a}{d}\Big) \D x,
  \end{split}
\end{equation}
where $\rho_Y$ denotes the density of $Y$.
We let $\lambda_h$ to stand for the operator norm of $L_h$ in $L^2(\nu )$,
\begin{equation}
  \label{def:lambda_h}
  \lambda_h \defeq \norm{L_h}_{L^2(\nu)\to L^2(\nu)}.
\end{equation}

The following proposition summarises some known properties of the operator $L_h$ as well
as the connection between $L_h$ and the critical height $h^*$.
\begin{proposition}[\cite{Szn15} Propositions 3.1, 3.3,  Corollary 4.5]
  \label{prop:L_h-chi-connection}
  For all $h\in\R$,
  $L_h$ is a self-adjoint, non-negative, Hilbert-Schmidt operator on
  $L^2(\nu)$,
  $\lambda_h$ is its simple eigenvalue and there
  exists a unique $\chi_h \geq 0$ with unit $L^2(\nu)$-norm, continuous,
  strictly positive on $[h, \infty)$, vanishing on $(-\infty,h)$, such that
  \begin{equation}
    L_h[\chi_h] = \lambda_h \chi_h.
  \end{equation}
  Additionally, the map $h \mapsto \lambda_h$ is a decreasing homeomorphism
  from $\R$ to $(0, d)$ and $h^*$ is the unique value in $\R$ such that
  $\lambda_{h^*} = 1$. Finally, for every $d\ge 2$,
  \begin{equation}
    0 < h^* < \infty.
  \end{equation}
\end{proposition}

Later we will need the following estimates on the norms of $L_h[f]$ which follow
from the hypercontractivity of the Ornstein-Uhlenbeck semigroup, see (3.14) in
\cite{Szn15} and (4.12) in \cite{AbaCer19}.
\begin{proposition}
  \label{prop:L_h-hypercontractivity}
  For every $f\in L^2(\nu)$, $h \in \R$, $1<p<\infty$ and
  $q \le (p-1)d^2 + 1$,
  \begin{equation}
    \label{eq:L_h-hypercontractivity_1}
    \norm[\big]{L_h[f]}_{L^q(\nu)}
    \leq \norm[\Big]{d E_Y\Big[
        f\Big(Y+\frac{\cdot}{d}\Big)\Big]}_{L^q(\nu)}
    \leq d\,\norm{f}_{L^p(\nu)}.
  \end{equation}
  In particular (taking $p=2$),
  \begin{equation}
    \norm[\big]{L_h[f]^k}_{L^2(\nu)}
    \leq d\,\norm{f}_{L^2(\nu)}^k \quad
    \text{ for all }  1\le k\le d^2 + 1.
  \end{equation}
\end{proposition}

The eigenfunctions $\chi_h$ of $L_h$ were studied more in detail in
\cite{AbaCer19}. We will need the following proposition describing their
behaviour. (Note that \cite{AbaCer19} considers $d$-regular trees, and
  thus $d$ in our setting corresponds to $d-1$ in \cite{AbaCer19}.)
\begin{proposition}[\cite{AbaCer19} Proposition 3.1]
  \phantomsection
  \label{prop:chi_bounds}
  \begin{enumerate}
    \item There exists $c > 0$ such that
    \begin{equation}
      \chi_h(a) \leq c a^{1-\log_d(\lambda_h)} \text{ for all }
      h\in \mathbb R \text{ and } a \geq d.
    \end{equation}
    \item For every $h\in\R$ there exists $c_h > 0$ such that
    \begin{equation}
      \label{eq:chi_lower_bound}
      \chi_h(a) \geq c_h a^{1-\log_d(\lambda_h)} \text{ for all } a \geq h.
    \end{equation}
  \end{enumerate}
\end{proposition}

Finally, we introduce the filtration
\begin{equation}
  \mathcal F_n \defeq \sigma\big(\phi_x : x \in \T^+,\ d(o,x) \leq n\big), \qquad
  n\ge 0,
\end{equation}
and recall from \cite[(3.35)]{Szn15}, that the $(\mathcal F_n)$-adapted
process $M^h = (M_n^h)_{n \geq 0}$ defined by
\begin{equation}
  \label{def:M_n}
  M_n^h \defeq \lambda_h^{-n} \sum_{x\in Z_n^h} \chi_h(\phi_x)
\end{equation}
is a non-negative martingale under $P$ as well as under every $P_a$,
$a\in \mathbb R$.

Throughout the paper we use the usual notation for the asymptotic
relation of two functions: For functions $f$ and $g$, we write
$f(s) \sim g(s)$ as $s \to s_0$ if $\lim_{s\to s_0} \frac{f(s)}{g(s)} = 1$,
and write $f(s) = o(g(s))$ as $s \to s_0$ if
$\lim_{s \to s_0} \frac{\abs{f(s)}}{g(s)} = 0$. We use $c,c',c_1,\dots$
to denote finite positive constants whose values may change from place to
place and which can only depend on $d$. The dependence of these constants
on additional parameters appears in the notation.

\section{Percolation probability at the critical height}
\label{sec:proof_theorem1}

In this section, we will show Theorem~\ref{thm:crit_BP_dies} which states
that there is no percolation at critical height $h^*$. Its
proof uses arguments that are rather common in the context of branching
processes and is given for sake of completeness. It exploits the fact
that the martingale $(M_n^h)_{n\geq 0}$ introduced in \eqref{def:M_n}
converges almost surely, which induces certain boundedness of the sizes
of the generations $Z_n^h$ (see \eqref{eq:Z_n^h}) as well as of the value
of the field on them. This is then enough to show the almost sure
finiteness of $\mathcal{C}_o^{h^*} \cap \T^+$.

To keep the notation simple, we often omit $h=h^*$ from the notation and
write,
e.g.,  $Z_n \defeq Z_n^{h^*}$, $M_n \defeq M_n^{h^*}$ and
$\chi = \chi_{h^*}$.
Let $\mathcal A$ be the event that $\mathcal{C}_o^{h^*} \cap \T^+$ has
infinite size,
\begin{equation}
  \mathcal A \defeq \set{\abs{\mathcal{C}_o^{h^*} \cap \T^+} = \infty},
\end{equation}
and let $\Phi_n\defeq \max_{x \in Z_n} \phi_x$ to be the maximum of the
field over $Z_n$ (with the convention that a maximum over the empty set
  is $-\infty$). For $H>0$, $N\in \mathbb N$ we define the events
\begin{equation}
  \mathcal{C}_H \defeq \set{\Phi_n \leq H \text{ for all } n \geq 0} \quad \text{and} \quad
  \mathcal D_N \defeq \set{\abs{Z_n} \leq N \text{ for all } n \geq 0}.
\end{equation}
We first show that for $H$ and $N$ large those events
are typical.

\begin{lemma}
  \label{lem:bounded_PhiZn}
  For every $a\ge h^*$ and $\varepsilon >0$ there is $H=H(a, \varepsilon )<\infty$ and
  $N= N(a,\varepsilon )<\infty$ so that
  \begin{align}
    \label{eq:bounded_Phin}
    P_a[\mathcal C_H] &\geq 1-\varepsilon, \\
    \label{eq:bounded_Zn}
    P_a[\mathcal D_N] &\geq 1-\varepsilon.
  \end{align}
\end{lemma}

\begin{proof}
  From the almost sure convergence
  of the non-negative martingale $M$, it follows that
  \begin{equation}
    \label{eq:Mn-convergence}
    \text{for every $\varepsilon >0$ there is $N<\infty$ such that }
    P_a\big[\sup_{n\ge 0} M_n \le N\big] \geq 1-\varepsilon.
  \end{equation}
  Indeed, assume that the statement does not hold. Then, there exists a
  $\varepsilon_0 > 0$ such that the events $A_k \defeq \{\sup M_n \ge k \}$
  satisfy $P_a[A_k] > \varepsilon_0$ for all $k\in \mathbb N$. Since
  $A_{k+1}\subseteq A_k$, this implies
  $P_a[\sup M_n = \infty] > \varepsilon_0$ which contradicts
  the almost sure convergence of $M$ to a finite limit.

  To prove \eqref{eq:bounded_Phin}, observe that
  $\lambda_{h^*}=1$ implies that
  $M_n = \sum_{x\in Z_n} \chi(\phi_x) \geq \chi(\Phi_n)$. Therefore,
  setting $H' = \inf \{\chi (h):h>H\}$ and
  using that $\Phi_n > H $ implies $\chi (\Phi_n) \ge H'$, we
  obtain
  \begin{equation}
    \label{eq:inclusions}
    \mathcal C_H = \big\{\sup \Phi_n \le H\big\}
    \supseteq \big\{\sup \chi (\Phi_n) < H'\big\}
    \supseteq \big\{\sup M_n < H'\big\}.
  \end{equation}
  By Proposition \ref{prop:chi_bounds} $\lim_{x\to\infty}\chi(x) = \infty$,
  and thus for $N$ as in \eqref{eq:Mn-convergence} there is $H'$ so that
  $H'\ge N+1$.  Estimate \eqref{eq:bounded_Phin} then follows from
  \eqref{eq:Mn-convergence} and~\eqref{eq:inclusions}.

  Estimate \eqref{eq:bounded_Zn} is proved similarly. By
  \eqref{eq:chi_lower_bound} there is $c>0$ such that
  $ c \leq \chi(h)$  for every $h \in [h^*, \infty)$. Therefore,
  $M_n \geq c \abs{Z_n}$, and thus
  $\mathcal D_N \supseteq \{{\sup M_n \le c N}\}$.
  Claim~\eqref{eq:Mn-convergence} then directly implies
  \eqref{eq:bounded_Zn}.
\end{proof}

We now argue that the events $\mathcal C_H$ and
$\mathcal D_N$ exclude the percolation event $\mathcal A$.

\begin{lemma}
  \label{lem:dead_bounded}
  $P_a[\mathcal A \cap \mathcal{C}_H \cap \mathcal D_N] = 0$
  for every $a\ge h^*$, $H\geq h^*$,  and  $N \geq 1$.
\end{lemma}

\begin{proof}
  For given $H \geq h^*$ and $N > 0$, let $\Theta_n$ be the event that
  up to generation $n$, no generation of $\mathcal{C}_o^{h^*}\cap \T^+$ exceeds $N$ and
  the field is bounded by $H$,
  \begin{equation}
    \Theta_{n} \defeq \{|Z_k|\leq N \text{ and } \Phi_k \leq H \text{ for all }
      k\leq n\}, \qquad n \geq 0,
  \end{equation}
  and let $\mathcal A_n$ be the event that the $n$-th generation is non-empty,
  \begin{equation}
    \mathcal A_n \defeq \{|Z_n|>0\}, \qquad n \geq 0.
  \end{equation}
  The sequences $\mathcal A_n$ and $\Theta_n$ are decreasing, with
  $\cap_{n\geq 0} \Theta_{n} = \mathcal{C}_H \cap \mathcal D_N$,
  $\cap_{n\geq 0} \mathcal A_n = \mathcal A$. Therefore,
  \begin{equation}
    \label{eq:P-limit}
    P_a[\mathcal A \cap \mathcal  C_H \cap \mathcal D_N]
    = \lim_{n\to\infty} P_a[\mathcal A_n \cap \Theta_n].
  \end{equation}

  We will show that this limit is zero.
  Conditionally on the event $\Theta_n \cap
  \mathcal A_n$, the number of particles in $Z_n$ is limited by
  $N$ and their types are bounded by $H$. Therefore, by the stochastic domination
  \eqref{eq:stoch_dom},  the conditional probability that $Z_{n+1}$ is
  empty can be bounded from below by the probability that $N$ independent
  particles of type $H$ have no descendants,
  \begin{equation}
    P_a[\mathcal A_{n+1}^c \mid  \mathcal A_{n}\cap\Theta_{n} ] \geq P_H[\abs{Z_1} = 0]^N
    \ge c > 0.
  \end{equation}
  As consequence, since
  $\mathcal A_{n+1}\cap \Theta_{n+1} \subset \mathcal A_n \cap \Theta_n$,
  \begin{equation}
    \frac{P_a[\mathcal A_{n+1} \cap \Theta_{n+1}]}{P_a[\mathcal A_{n} \cap \Theta_{n}]}
    = P_a[\mathcal A_{n+1} \cap \Theta_{n+1}\mid \mathcal A_{n} \cap \Theta_{n}] \le
    P_a[\mathcal A_{n+1} \mid  \mathcal A_{n}\cap\Theta_{n} ]  \le 1-c.
  \end{equation}
  Applying this bound inductively proves that
  the limit on the right-hand side of \eqref{eq:P-limit} is zero,
  completing the proof.
\end{proof}

With help of Lemmas~\ref{lem:bounded_PhiZn},
\ref{lem:dead_bounded}, it is straightforward to complete the proof of
Theorem~\ref{thm:crit_BP_dies}.

\begin{proof}[Proof of Theorem \ref{thm:crit_BP_dies}]
  By Lemma~\ref{lem:bounded_PhiZn}, for an arbitrary
  $\varepsilon > 0$ and $a\in \mathbb R$ there is $H$ and
  $N$ such that $P_a[\mathcal{C}_H] \geq 1-\varepsilon$ and $P_a[\mathcal D_N] \geq 1-\varepsilon$. Therefore,
  using also Lemma~\ref{lem:dead_bounded}
  \begin{equation}
    0 =P_a[\mathcal A \cap \mathcal{C}_H \cap \mathcal D_N] \geq P_a[\mathcal A] - 2\varepsilon.
  \end{equation}
  Since
  $\varepsilon $ is arbitrary, this implies $P_a[\mathcal A]=0$ as required.
  The second claim of the theorem follows from the equality
  \begin{equation}
    P[\mathcal A] = \int P_a[\mathcal A]\, \nu (\D a) =0,
  \end{equation}
  which holds due to \eqref{eq:BP-representation}.
\end{proof}

\section{Distribution of the size of the critical cluster}
\label{sec:proof_theorem2}

In this section we prove Theorem \ref{thm:cluster_size_tail} describing
the asymptotic behaviour of the size of the connected clusters
$\abs{\mathcal{C}_o^{h^*}}$ and $\abs{\mathcal{C}_o^{h^*} \cap \T^+}$ in
the critical case $h = h^*$.

To this end we denote by $T$ the total size of $\mathcal C_o^{h^*}$ restricted
to the forward tree,
\begin{equation}
  T \defeq \abs{\mathcal{C}_o^{h^*} \cap \T^+},
\end{equation}
and let $\mathcal L_{a}(s)$ be its
Laplace transform under $P_a$,
\begin{equation}
  \mathcal L_{a}(s) \defeq E_a[e^{-s T}],\quad a \in \R, s \geq 0.
\end{equation}

The proof of Theorem~\ref{thm:cluster_size_tail} is based on the
following classical Tauberian theorem, that connects the asymptotic
behaviour of the cumulative distribution function of a random variable at
infinity and its Laplace transform near zero.
\begin{proposition}[Corollary 8.1.7, \cite{BinGolTeu89}]
  \label{prop:tauberian}
  Let $X$ be a non-negative random variable with cumulative distribution
  function $F$ and Laplace transform $\mathcal L(s)=E[e^{-sX}]$.  For $0 \leq \alpha < 1$
  and a function $\ell:[0,\infty) \to [0,\infty)$ slowly varying at $\infty$ the
  following are equivalent:
  \begin{enumerate}
    \item $1 - \mathcal L(s) \sim \Gamma (1-\alpha )s^\alpha \ell(1/s)$ as $s \to 0^+$,
    \item $1-F(t) \sim  t^{-\alpha }\ell(t)$ as $t \to \infty$.
  \end{enumerate}
\end{proposition}

In view of this proposition, to show Theorem \ref{thm:cluster_size_tail}
we first need to control the asymptotic behaviour of $1-\mathcal L_a(s)$.

\begin{proposition}
  \label{prop:L-convergence}
  For every $a \geq h^*$,
  \begin{equation}
    \lim_{s \downarrow 0}s^{-1/2}\big(1-\mathcal L_{a}(s)\big) = C_1 \Gamma(1/2) \chi(a),
  \end{equation}
  where $C_1$ was defined  in \eqref{eq:Cx}.
\end{proposition}

We start with some basic observations and definitions that will eventually
lead to the proof of this proposition.
By Theorem~\ref{thm:crit_BP_dies},  $P_a[T =\infty] = 0$ for every $a \in \R$,
and thus
\begin{equation}
  \label{eq:La_to_one}
  \lim_{s \downarrow 0}\big(1-\mathcal L_{a}(s)\big) = 0.
\end{equation}
Moreover, the Laplace transform $\mathcal L_a(s)$ satisfies the recursive equation
\begin{equation}
  \label{eq:self-reference}
  \mathcal L_{a}(s) = \begin{cases}
    e^{-s}\, \big(E_Y[\mathcal L_{\frac{a}{d} + Y}(s)]\big)^d, \quad&\text{if }
    a \geq h^*, \\
    1, &\text{if }a < h^*,
  \end{cases}
\end{equation}
where, as in \eqref{eq:L_h}, $Y\sim \mathcal N(0,\sigma_Y^2)$. To see
this in the case $a\ge h^*$ (the other case is trivial), it is sufficient
to write  $T=1+T_1+\dots+T_d$, where $T_i$ is the size of the
intersection of $\mathcal C_0^{h^*}$ with the sub-tree of the $i$-th
neighbour $x_i$ of the root $o$, and observe that $T_1,\dots,T_d$ are
conditionally independent given $\phi_o=a$ with respective Laplace
transforms
\begin{equation}
  E_a[e^{-sT_i}]=E_a\big[E_a[e^{-sT_i}\mid \phi_{x_i}]\big]
  = E_a[\mathcal L_{\phi_{x_i}}(s)]
  =E_Y[\mathcal L_{\frac{a}{d} + Y}(s)],
\end{equation}
where the last equality uses the branching process representation
\eqref{eq:BP-representation} of $\phi$.

We further set
\begin{equation}
  \gamma_s(a) \defeq 1 - \mathcal L_{a}(s),
\end{equation}
and note that $\gamma_s(a) = 0$ for $a < h^*$, $\gamma_s(a)\in [0,1]$ for
every $s \ge 0$ and $a \in \R$, and therefore for every $s\ge0$,
$\gamma_s \in L^2(\nu)$.  By \eqref{eq:self-reference},  using the
operator $L = L_{h^*}$ from~\eqref{eq:L_h}, for $a\ge h^*$,
\begin{equation}
  1-\gamma_s(a)
  = e^{-s}E_Y\Big[1-\gamma_s\Big(\frac ad + Y\Big)\Big]^d
  = e^{-s}\Big(1-\frac{1}{d}L[\gamma_s](a)\Big)^d.
\end{equation}
Rearranging this equality implies that for $a \geq h^*$ and $s \ge 0$,
\begin{equation}
  \label{eq:base_equation}
  1-e^{-s} - \gamma_s(a) + e^{-s}L[\gamma_s](a) =
  e^{-s}f(L[\gamma_s](a)),
\end{equation}
where the function $f:[0,d] \to \R$ is defined by
\begin{equation}
  \label{eq:f_definition}
  f(x) = f_d(x) \defeq  \Big(1-\frac{x}{d}\Big)^d - 1 + x
  = \sum_{k=2}^d \binom{d}{k} (-1)^k\Big(\frac{x}{d}\Big)^k.
\end{equation}
Equation \eqref{eq:base_equation} will be the starting point for several
proofs that follow.

We continue with a simple observation about the function $f$.

\begin{lemma}
  \label{lem:f-bound}
  For any $d \geq 2$ there are constants $c_1$, $c_2$ such that for all $x\in [0,d]$,
  \begin{equation}
    c_1 x^2 \leq f(x) \leq c_2 x^2.
  \end{equation}
\end{lemma}

\begin{proof}
  From \eqref{eq:f_definition} it is easy to see that  $f$ is smooth,
  strictly convex on $[0,d]$ with $f(0)=f'(0)=0$ and $f''(0)=c >0$. It
  follows that $c x^2/2 \le f(x) \le 2 c x^2$ for $x$ in a certain interval
  $[0,\varepsilon ]$, and that $f$ is strictly positive and bounded on
  $(\varepsilon ,d]$. From these two facts the lemma easily follows.
\end{proof}

In the remainder of this section we exclusively work in $L^2(\nu)$, and
denote by $\norm{\,\cdot\,}$ and $\inp\cdot\cdot$ the corresponding norm
and scalar product. Since $L$ is self-adjoint (see
  Proposition~\ref{prop:L_h-chi-connection}), $L^2(\nu)$ has an
orthonormal basis consisting of the eigenfunctions $\set{e_k}_{k\geq 1}$
of $L$ corresponding to the eigenvalues $\set{\lambda_k}_{k\geq 1}$. Since
$h= h^*$, by Proposition~\ref{prop:L_h-chi-connection} we may assume that
$1=\lambda_1 > \abs{\lambda_2} \geq \abs{\lambda_3} \geq \dots \geq 0$
with $\lambda_k \to 0$ as $k\to \infty$, and also
$e_1 = \chi$. Therefore
\begin{equation}
  \gamma_s = \sum_{k\geq 1} a_k(s) e_k, \qquad
  \text{with}\qquad
  a_k(s) \defeq \inp{\gamma_s}{e_k}.
\end{equation}
By considering the first summand separately, we write $\gamma_s$ as
\begin{equation}
  \label{eq:decomp_y}
  \gamma_s = \alpha_s + \beta_s,
  \qquad \text{with}\qquad
  \alpha_s \defeq a_1(s) \chi \text { and }\beta_s \defeq \sum_{k\geq 2}a_k(s)e_k.
\end{equation}
Observe that
\begin{equation}
  \label{eq:decomp_Ly}
  L[\gamma_s] = \sum_{k\geq 1} \lambda_k a_k(s) e_k = \alpha_s + \sum_{k\geq 2} \lambda_k a_k(s) e_k
  = \alpha_s + L [\beta_s].
\end{equation}
and thus
\begin{equation}
  \label{eq:inp_gamma_L_gamma_is_zero}
  \inp{\gamma_s - L[\gamma_s]}{\chi} = 0.
\end{equation}
Since $\gamma_s(a)\in[0,1]$, the definition \eqref{eq:L_h} of $L$ and
Lemma~\ref{lem:f-bound} imply that
\begin{equation}
  \label{eq:LfLbounds}
  0\le L[\gamma_s] \le d
  \qquad\text{and}\qquad
  0\le f(L[\gamma_s])\le c_2 L[\gamma_s]^2 \le c_2 d^2.
\end{equation}
From the pointwise convergence \eqref{eq:La_to_one} of $\gamma_s$ to $0$,
using successively
the dominated convergence theorem and the continuity
of $L$, it follows that
\begin{equation}
  \label{eq:convergence_to_0_in_L2}
  \lim_{s\downarrow 0}\gamma_s
  =\lim_{s\downarrow 0}L[\gamma_s]
  =\lim_{s\downarrow 0}f(L[\gamma_s]) = 0 \quad \text{in } L^2(\nu).
\end{equation}
In particular $a_k(s) \to 0$ as $s\downarrow  0$ for all $k\ge 1$.
Finally,  since $\norm{\gamma_s -
  L[\gamma_s]}^2 = \sum_{k\geq 2} (1-\lambda_k)^2 a_k(s)^2 \geq
\sum_{k\geq 2} (1-\abs{\lambda_2})^2 a_k(s)^2 = (1-\abs{\lambda_2})^2
\norm{\beta_s}^2$, it holds that
\begin{equation}
  \label{eq:norm_beta_bound}
  \norm{\beta_s} \leq \frac{1}{1-\abs{\lambda_2}} \norm{\gamma_s - L[\gamma_s]}.
\end{equation}

The following three lemmas are the main preparatory steps for the proof of
Proposition~\ref{prop:L-convergence}. They together show that
$\alpha_s$ dominates $\beta_s$ in norm and then estimate
$\alpha_s$ precisely.

\begin{lemma}
  \label{lem:as-bound}
  There is a constant $c<\infty$ such that $\norm{\alpha_s}^2 \leq c s$ for
  all $s$ small enough.
\end{lemma}

\begin{proof}
  Noting that $\chi = 0$ on $(-\infty,h^*)$ and applying $\inp{\cdot}{\chi}$ on both
  sides of \eqref{eq:base_equation} yields
  \begin{equation}
    \label{eq:as-bound}
    (1-e^{-s})\big(\inp{1}{\chi} - a_1(s)\big)
    = e^{-s} \inp[\big]{f(L[\gamma_s])}{\chi}
    \quad \text{for } s \ge 0.
  \end{equation}
  By \eqref{eq:chi_lower_bound}, $\chi > c >0$ on $[h^*, \infty)$.
  Therefore, using also Lemma \ref{lem:f-bound}, the right-hand side of
  \eqref{eq:as-bound} satisfies for $s \le 1 $
  \begin{equation}
    e^{-s}\inp{f(L[\gamma_s])}{\chi} \geq c
    \inp{L[\gamma_s]^2}{\chi} \geq c' \norm{L[\gamma_s]}^2\ge
    c' \norm{\alpha_s}^2 = c' \alpha_1(s)^2,
  \end{equation}
  where the last inequality follows from the orthogonal decomposition \eqref{eq:decomp_Ly}.
  Together with \eqref{eq:as-bound}, this gives
  \begin{equation}
    (1-e^{-s})(\inp{1}{\chi} - a_1(s)) \geq c'\alpha_1(s)^2
    \quad \text{for $s\le 1$}.
  \end{equation}
  Since, as $s\downarrow 0$, $(1-e^{-s}) \sim s$ and $a_1(s)\to 0$
  this finishes the proof.
\end{proof}

\begin{lemma}
  \label{lem:bz-bound}
  There is a constant $c<\infty$ such that $\norm{\beta_s} \leq c s$ for $s$ small
  enough.
\end{lemma}

\begin{proof}
  We rearrange equation \eqref{eq:base_equation} to obtain
  \begin{equation}
    \gamma_s - L[\gamma_s] = (1-e^{-s}) (1 - L[\gamma_s]) - e^{-s} f(L[\gamma_s])
    \quad \text{on} \quad [h^*, \infty).
  \end{equation}
  Since the left-hand side is identically zero on $(-\infty, h^*)$,
  taking norms yields
  \begin{equation}
    \norm{\gamma_s - L[\gamma_s]} \leq (1-e^{-s}) \norm{1 -
      L[\gamma_s]} + \norm{f(L[\gamma_s])}.
  \end{equation}
  Using \eqref{eq:LfLbounds}, $1-e^{-s} \le s$, and
  Lemma~\ref{lem:f-bound}, this implies
  \begin{equation}
    \label{eq:bz-a}
    \norm{\gamma_s - L[\gamma_s]} \leq (1+d) s + c \norm{L[\gamma_s]^2}.
  \end{equation}
  The norm on the right-hand side can be bounded using
  Proposition~\ref{prop:L_h-hypercontractivity} and the inequality
  $(a+b)^2\le 2 a^2+ 2 b^2$,
  \begin{equation}
    \label{eq:bz-b}
    \norm{L[\gamma_s]^2} \leq
    d \norm{\gamma_s}^2 \leq
    d (\norm{\alpha_s} + \norm{\beta_s})^2 \leq
    2d(\norm{\alpha_s}^2 + \norm{\beta_s}^2).
  \end{equation}
  Combining \eqref{eq:bz-a}, \eqref{eq:bz-b} with
  \eqref{eq:norm_beta_bound}, we obtain that for
  a constant $c<\infty$ and all $s > 0$ small enough
  \begin{equation}
    \norm{\beta_s} \leq c s + c \norm{\alpha_s}^2 + c \norm{\beta_s}^2.
  \end{equation}
  By \eqref{eq:convergence_to_0_in_L2}, $\lim_{s\to 0}\norm{\gamma_s} = 0$
  and thus $\lim_{s\to 0 } \norm{\beta_s} = 0$ as well. The claim of the
  lemma then follows easily from
  Lemma~\ref{lem:as-bound}.
\end{proof}

\begin{lemma}
  \label{lem:as_limit}
  It holds that $\lim_{s\downarrow 0}s^{-1}{a_1(s)}^2 = (C_1 \Gamma(1/2))^2$, where
  $C_1$ was defined in \eqref{eq:Cx}.
\end{lemma}

\begin{proof}
  We start by proving the estimate
  \begin{equation}
    \label{eq:as_limit_1}
    \norm[\bigg]{(1-e^{-s}) 1_{[h^*, \infty)}-\gamma_s+L[\gamma_s] -
      \frac{1}{d^2}\binom{d}{2} L[\alpha_s]^2} \leq c s^{3/2}
  \end{equation}
  holding for some constant $c<\infty$ and all $s$ small enough:
  Rearranging \eqref{eq:base_equation}
  and subtracting $\frac{1}{d^2} \binom{d}{2} L[\alpha_s]^2$ on both
  sides shows that, on $[h^*,\infty)$,
  \begin{multline}
    (1-e^{-s}) - (\gamma_s - L[\gamma_s]) - \frac{1}{d^2}\binom{d}{2} L[\alpha_s]^2 \\
    =  (1-e^{-s})\big(L[\gamma_s] - f(L[\gamma_s])\big)
    + f(L[\gamma_s])- \frac{1}{d^2}\binom{d}{2} L[\alpha_s]^2.
  \end{multline}
  After taking norms, using again that $s \sim (1-e^{-s})$ as
  $s \downarrow 0$, this implies that
  \begin{multline}
    \label{eq:as_limit_2}
    \norm[\bigg]{(1-e^{-s})1_{[h^*, \infty)} -
      (\gamma_s - L[\gamma_s]) - \frac{1}{d^2}\binom{d}{2} L[\alpha_s]^2} \\
    \leq  c s  \norm{L[\gamma_s]} + c s\norm{f(L[\gamma_s])}
    +  \norm[\bigg]{f(L[\gamma_s])- \frac{1}{d^2}\binom{d}{2} L[\alpha_s]^2}
  \end{multline}
  for some constant $c$ and $s$ small enough.  By Lemmas~\ref{lem:as-bound} and
  \ref{lem:bz-bound},
  $\norm{L[\gamma_s]} \leq  \norm{\gamma_s}
  \leq \norm{\alpha_s} + \norm{\beta_s} \leq c s^{1/2}$. Further, by
  Lemma~\ref{lem:f-bound}--\ref{lem:bz-bound} and \eqref{eq:bz-b},
  $\norm{f(L[\gamma_s])} \leq \norm{cL[\gamma_s]^2}
  \le c\norm{\alpha_s}^2 + c \norm{\beta_s}^2 \leq c s$. Hence, to show
  \eqref{eq:as_limit_1}, it remains to bound the
  last summand in \eqref{eq:as_limit_2} by $c s^{3/2}$.  By the definition
  \eqref{eq:f_definition} of $f$,
  \begin{equation}
    f(L[\gamma_s]) - \frac{1}{d^2}\binom{d}{2} L[\alpha_s]^2 =
    2\frac{1}{d^2}\binom{d}{2} L[\alpha_s] L[\beta_s] + \frac{1}{d^2}\binom{d}{2}
    L[\beta_s]^2 + \sum_{k=3}^{d}c_k L[\gamma_s]^k
  \end{equation}
  for some $c_k\in (0,\infty)$. Hence,  after taking
  the norm,
  \begin{equation}
    \label{eq:normfbound}
    \norm[\bigg]{f(L[\gamma_s]) - \frac{1}{d^2}\binom{d}{2} L[\alpha_s]^2}
    \leq  c \norm{L[\alpha_s] L[\beta_s]} + c \norm{L[\beta_s]^2} + c \sum_{k=3}^{d} \norm{L[\gamma_s]^k},
  \end{equation}
  for some constant $c > 0$. By the Cauchy-Schwarz inequality,
  Proposition~\ref{prop:L_h-hypercontractivity} and
  Lemmas~\ref{lem:as-bound}, \ref{lem:bz-bound},
  \begin{equation}
    \begin{split}
      \norm{L[\alpha_s]
        L[\beta_s]}^2 &= \inp{L[\alpha_s]L[\beta_s]}{L[\alpha_s]L[\beta_s]} =
      \inp{L[\alpha_s]^2}{L[\beta_s]^2}
      \\& \leq
      \norm{L[\alpha_s]^2}\norm{L[\beta_s]^2}
      \le  c\norm{\alpha_s}^2 \norm{\beta_s}^2 \le c s^{3/2},
      \\
      \norm{L[\beta_s]^2} &\leq c \norm{\beta_s}^2 \le cs^2,
    \end{split}
  \end{equation}
  and for $3\le k \le d$, by the same arguments,
  \begin{equation}
    \norm{L[\gamma_s]^k}
    \leq c \norm{\gamma_s}^k \le c(\norm{\alpha_s}+\norm{\beta_s})^k
    \leq c 2^{k-1} (\norm{\alpha_s}^k + \norm{\beta_s}^k) \le c s^{3/2}.
  \end{equation}
  This proves that the third summand on the right-hand side of
  \eqref{eq:as_limit_2} is bounded by $cs^{3/2}$ and thus completes the
  proof of \eqref{eq:as_limit_1}.

  We can now show the lemma.
  From \eqref{eq:as_limit_1}, using
  $\lim_{s\downarrow 0}s^{-1}(1-e^{-s}) = 1$, it
  easily follows that
  \begin{equation}
    \lim_{s \downarrow 0} \bigg(\frac{1}{s}(\gamma_s-L[\gamma_s])
      + \frac{1}{d^2}\binom{d}{2}\frac{L[\alpha_s]^2}{s}
    \bigg)
      = 1_{[h^*,\infty)}
    \quad \text{in $L^2(\nu)$}.
  \end{equation}
  Since  $\chi = 0$ on
  $(-\infty, h)$ and $L[\alpha_s] = \alpha_s= a_1(s) \chi$, this implies that
  \begin{equation}
    \begin{split}
    \inp{1}{\chi} &= \inp[\bigg]{\lim_{s \downarrow 0} \bigg(\frac{1}{s}(\gamma_s-L[\gamma_s]) +
        \frac{1}{d^2}\binom{d}{2}\frac{\alpha_s^2}{s}\bigg)}{\chi}
    \\&= \lim_{s\downarrow 0} \bigg( \inp[\Big]{\frac{1}{s}(\gamma_s
            -L[\gamma_s])}{\chi}+
      \frac{1}{d^2}\inp[\Big]{\binom{d}{2} \frac{\alpha_s^2}{s}}{\chi}\bigg)
    \\&= \frac{d-1}{2d} \inp{\chi^2}{\chi} \lim_{s\downarrow 0}  \frac{a_1(s)^2}{s},
    \end{split}
  \end{equation}
  where in the last equality we used
  $\inp{\gamma_s -L[\gamma_s]}{\chi}= 0$, by
  \eqref{eq:inp_gamma_L_gamma_is_zero}. The claim of the lemma then follows.
\end{proof}

We now have all ingredients to give the proof of
Proposition~\ref{prop:L-convergence}, directly followed by the proof of
Theorem~\ref{thm:cluster_size_tail}.

\begin{proof}[Proof of Proposition~\ref{prop:L-convergence}]
  By the Lemmas \ref{lem:bz-bound} and \ref{lem:as_limit},
  \begin{equation}
    \label{eqn:aaaa}
    \lim_{s\downarrow 0}\frac{\gamma_s}{\sqrt{s}} =
    \lim_{s\downarrow 0}\Big(\frac{a_1(s) \chi}{\sqrt{s}} +
      \frac{\beta_s}{\sqrt{s}}\Big) =
    C_1 \Gamma(1/2) \chi \qquad \text{in } L^2(\nu).
  \end{equation}
  By the stochastic domination~\eqref{eq:stoch_dom}, the function
  $a\mapsto\gamma_s(a)$ is
  increasing for any $s>0$, and by Proposition~\ref{prop:L_h-chi-connection}, the
  limit function $C_1 \Gamma(1/2) \chi$ is continuous on
  $[h^*, \infty)$. This implies that the convergence in \eqref{eqn:aaaa}
  is pointwise as well.
\end{proof}

\begin{proof}[Proof of Theorem~\ref{thm:cluster_size_tail}]
  Claim \eqref{eq:thm_cst_1} follows directly from Propositions~\ref{prop:tauberian}
  and~\ref{prop:L-convergence}.

  To prove
  \eqref{eq:thm_cst_2}, let $\tilde{\mathcal L}_a$ be the Laplace transform of $\abs{\mathcal{C}_o^{h^*}}$
  under $P_a$,
  \begin{equation}
    \tilde{\mathcal L}_{a}(s) \defeq E_a\big[e^{-s
        \abs{\mathcal{C}_o^{h^*}}}\big],\qquad a \in \R, s \geq 0.
  \end{equation}
  Using the same arguments as in the proof of the
  recursion property \eqref{eq:self-reference}, it follows that
  \begin{equation}
    \tilde{\mathcal L}_a(s) = e^{-s} \big(E_Y[\mathcal L_{\frac{a}{d} + Y}
        (s)]\big)^{d+1},
    \qquad \text{for } s > 0, a \geq h^*,
  \end{equation}
  which together with \eqref{eq:self-reference} yields
  \begin{equation}
    \label{eq:1-La_Tbar_equation}
     \tilde{\mathcal L}_a(s) = e^{s/d} \mathcal L_a(s)^{(d+1)/d} \qquad \text{for
    }s>0, a\geq h^*.
  \end{equation}
  Using Proposition~\ref{prop:L-convergence}, this implies that, as
  $s\downarrow 0$,
  \begin{equation}
    \begin{split}
      \tilde{\mathcal L}_a(s)
      &= \Big(1-\frac sd+o(s)\Big)
      \Big(1-C_1 \Gamma (1/2) \chi (a) s^{1/2} + o(s^{1/2})\Big)^{(d+1)/d}
      \\&= 1- \frac{d+1}{d}\, C_1\Gamma(1/2) \chi (a)s^{1/2} +o(s^{1/2}).
    \end{split}
  \end{equation}
  Claim~\eqref{eq:thm_cst_2} then follows by another application of
  Proposition~\ref{prop:tauberian}.
\end{proof}

\section{Behaviour of the connectivity for near-critical level set percolation}
\label{sec:proof_theorem3}

In this section we prove Theorem \ref{thm:percolation_prob} which
describes the asymptotic behaviour of the percolation probabilities
$\eta(h,a)$ and $\eta^+(h,a)$ for fixed $a\in\R$ as $h$ approaches $h^*$
from below.

The proof is based on a careful analysis of the functional
equation for $\eta^+$ that was proved in \cite{AbaCer19} and that we
recall in the next proposition.

\begin{proposition}[\cite{AbaCer19} Theorem 4.1]
  For every $h\in \R$, the forward percolation probability
  $\eta_h^+ \defeq \eta^+(h, \cdot)$ solves the
  functional equation
  \begin{equation}
    \label{eq:percolation-prob}
    f(a) = 1_{[ h,\infty)}(a)\Big( 1 -
      \big(1- d^{-1} L_h[f]\big)^d\Big),
    \quad a \in \R.
  \end{equation}
  In addition, the only two solutions of
  \eqref{eq:percolation-prob} in the set
  \begin{equation}
    \mathcal{S}_h \defeq \set{f \in L^2(\nu) : 0 \leq f \leq 1
      \text{ and } f = 0 \text { on } (-\infty, h)}
  \end{equation}
  are the constant function $f=0$ and $\eta^+_h$. For $h\ge h^*$ these
  two solutions coincide and for $h<h^*$ they are distinct.
\end{proposition}

The last claim of this proposition together with the continuity of the
percolation functions (cf.~Corollary~\ref{cor:continuity}) implies that
the solution set to \eqref{eq:percolation-prob} has a bifurcation at the
critical point $h^*$. Therefore, in order to describe the behaviour of
$\eta^+_h$ as $h\uparrow h_*$, we will analyse the solution set around
this bifurcation.

Our main tool will be the theorem on transcritical bifurcations on
general Banach spaces, stated as Proposition~\ref{prop:bifurcation1}
below. To introduce this theorem we need more notation. For Banach spaces
$X, Y$, let $B(X, Y)$ be the space of bounded linear operators from $X$
to $Y$. For a function $F: X \to Y$,  we use $DF: X \to B(X, Y)$ to
denote its Fréchet derivative and $DF(x): X \to Y$ its Fréchet derivative
evaluated at point $x \in X$. If $T$ is an open interval in $\mathbb R$
and $G: T \times X \to Y$, then we use $D_x G(t,x): X\to Y$ and
$D_t G(t, x)\in Y$ to denote the partial Fréchet derivative
in the $x$ and $t$ direction, evaluated at point $(t, x)$. Similarly,
$D_{xx}G$ or $D_{xt}G$ denote the respective second partial Fréchet
derivatives. Finally, we use $N(F)$ and $R(F)$ to denote the kernel and
the range of a linear functional $F$.

\begin{proposition}[\cite{CRANDALL}, Theorems 1.7 and 1.18]
  \label{prop:bifurcation1}
  Let $X$, $Y$ be Banach spaces, $V$ a neighbourhood of $0$ in $X$,
  $I = (t_0-1, t_0+1)$ for some $t_0 \in \R$. Assume that
  a (non-linear) functional $F:I \times V \rightarrow Y$ satisfies:
  \begin{enumerate}[label=(\alph*)]
    \item $F(t, 0) = 0$ for $t \in I$,
    \item The partial derivatives $D_t F$, $D_x F$ and $D_{tx}F$ exist and are continuous,
    \item $N(D_xF(t_0,0)) = \spn{\set{x_0}}$ for a $x_0 \in X$, and
    $\codim R(D_xF(t_0,0)) = 1$,
    \item $D_{tx} F(t_0, 0)(x_0) \not\in R(D_xF(t_0,0))$, where $x_0$ is given in (c).
  \end{enumerate}
  Then for any complement $Z$ of $x_0$ in $X$ (i.e.,~for any subspace $Z$ of
    $X$ with $Z \oplus \spn\set{x_0} = X$) there is a neighbourhood $U$ of
  $(t_0,0)$ in $\R \times X$, an interval $(-a, a)$, and continuous functions
  $\phi:(-a,a) \to I$, $\psi: (-a,a) \to Z$ such that $\phi(0) = t_0$, $\psi(0) =
  0$ and
  \begin{equation}
    F^{-1}(0) \cap U = \set{(\phi(\alpha), \alpha x_0  + \alpha \psi(\alpha)) : \abs{\alpha} < a}
    \cup \set{(t, 0) : (t, 0) \in U}.
  \end{equation}
  If, in addition to (a)--(d), $D_{xx} F$ is continuous, then the functions $\phi$ and $\psi$ have a
  continuous derivative with respect to $\alpha$ and
  \begin{equation}
    \label{eq:bifurcation2}
    \frac{1}{2} D_{xx}F(t_0,0)(x_0, x_0) + D_xF(t_0,0)(\psi'(0)) +
    \phi'(0)D_{tx}F(t_0,0)(x_0) = 0.
  \end{equation}
\end{proposition}

One of the main difficulties in applying this proposition to our
situation is to choose  suitable spaces $X$ and $Y$ where its conditions
can be verified. We start by shifting the functions $\eta_h^{+}$ so that
they have a common zero set. To this end, let $\theta_a$ be the usual
shift operator acting on $f:\mathbb R\to \mathbb R$ by
$\theta_a f(x) = f(x+a)$, and define
$\widetilde{\eta}_h \defeq \theta_h\eta_h^+$. Note that, for $h<h^*$,
$\widetilde \eta_h(a) >0$ iff $a\in [0,\infty)$. For $h\in \mathbb R$,
let $H_h$ be an operator defined by
\begin{equation}
  \label{eq:HLrelation}
  H_h[f] =d^{-1} \theta_h  L_h[\theta^{-1}_h f].
\end{equation}
Using the definition~\eqref{eq:L_h} of $L_h$, after an easy computation,
this operator can be written more explicitly:
\begin{equation}
  \label{eq:Hexpl}
  H_h[f](a) =
  \begin{cases}
    \int_0^\infty f(x)\rho_Y\big(x - \frac{a}{d} +
      \frac{d-1}{d}\, h\big) \D x,\quad &\text{when $a\ge 0$},
    \\ 0,&\text{otherwise,}
  \end{cases}
\end{equation}
where $\rho_Y$ denotes the centred Gaussian density with variance
$\sigma_Y^2$.
This notation allows to rewrite equation
\eqref{eq:percolation-prob} in terms of $\widetilde{\eta}_h$ as
\begin{equation}
  \label{eq:shifted_percolation_eq}
  0 = - \widetilde{\eta}_h(a) + 1_{[0, \infty)}(a)\Big(1-
  \big(1 - H_h[\widetilde{\eta}_h](a)\big)^d\Big), \qquad  a \in \R.
\end{equation}
Finally, let $\nu^*$ be a Gaussian measure obtained from the Gaussian
measure $\nu $ (see above \eqref{eq:L_h}) by shifting it by $h^*$, that is
the corresponding densities satisfy
$\rho_{\nu^*}  = \theta_{h^*} \rho_\nu $.

In view of \eqref{eq:shifted_percolation_eq}, to prove Theorem~\ref{thm:percolation_prob}, we
will show that Proposition~\ref{prop:bifurcation1} is applicable to
\begin{equation}
  \label{eq:F}
  F(h, f) = -f + 1_{[0, \infty)} \Big(1 -  \big( 1 - H_h[f]\big)^d\Big),
\end{equation}
viewed as a map from $I \times L^2(\nu^*)$ to $L^2(\nu^*)$,
with $I = (h^*-1,h^*+1)$.
Showing the applicability of Proposition \ref{prop:bifurcation1} is
divided into multiple steps.  First, we prove that $F$ is indeed a map
from $I \times L^2(\nu^*)$ to $L^2(\nu^*)$. Then we
compute the necessary partial Fréchet
derivatives, and finally we verify the remaining
assumptions of the proposition. Before starting with this programme,
we state a simple estimate that will later be useful several times.

\begin{lemma}
  \label{lem:phi_dominant}
  Let $\rho$ be any centred Gaussian density.  Then there exist
  constants $r,c\in (0,\infty)$ such that for all $x \in \R$ and $\abs s\in (-1,1)
  \setminus \set{0}$,
  \begin{equation}
    \abs[\bigg]{\frac{\rho(x + s) - \rho(x)}{s}}
    \leq c \big(\rho(x+r) + \rho(x-r)\big).
  \end{equation}
\end{lemma}

\begin{proof}
  By Taylor's theorem, $\rho(x+s) = \rho(x) + \rho'(x)s +
  \frac{1}{2}\rho''(\xi_{x,s})s^2$, for some $\xi_{x,s}$ between $x$ and $x+s$.
  Therefore, for $\abs s\in(0,1)$,
  \begin{equation}
    \label{eq:phi_dominant_1}
    \abs[\bigg]{\frac{\rho(x + s) - \rho(x)}{s}} \leq \abs{\rho'(x)} + \abs{\rho''(\xi_{x,s})}.
  \end{equation}
  Since $\rho'(x) = P_1(x)\rho (x)$ and $\rho''(x) = P_2(x)\rho(x)$ for
  some polynomials $P_1$, $P_2$, it follows easily that there is
  $x_0<\infty$ such that
  $\abs{\rho'(x)} + \abs{\rho''(\xi_{x,s})} \le (\rho (x+2)+ \rho(x-2))$
  for all $x\notin[-x_0,x_0]$ and
  $s\in [-1,1]$. Finally, since $\rho >0$ on $\mathbb R$, we can made the
  last inequality valid on whole $\mathbb R$ by multiplying the right-hand
  side by a sufficiently large constant $c$.
\end{proof}

We now show that $F$ maps $I\times L^2(\nu^*)$ to $L^2(\nu^*)$. To this
end it is enough
to prove that for $f\in L^2(\nu^*)$ and $h \in \R$,
$H_h[f] \in L^{2d}(\nu^*)$. The following lemma shows a little bit more,
as it will be needed in the later proofs, and also proves the continuity
of $h\mapsto H_h$.

\begin{lemma}
  \phantomsection
  \label{lem:image_of_H}
  \begin{enumerate}
    \item If $h,s \in \R$ and $f\in L^2(\nu^*)$, then
    \begin{equation}
      \label{eq:Hboundedoperator}
      \norm{\theta_s H_h[f]}_{L^{2d}(\nu^*)} \le C_{h,s}
      \norm{f}_{L^2(\nu^*)}.
    \end{equation}
    In particular, $\theta_s H_h[f] \in L^{2d}(\nu^*)$.

    \item The function $h\mapsto H_h$ from $I$ to
    $B(L^2(\nu^*), L^{2d}(\nu^*))$ is (strongly) continuous.
  \end{enumerate}
\end{lemma}

\begin{proof}
  (a) For $f\in L^2(\nu^*)$, let $g = \theta_{h^*}^{-1} f \in L^2(\nu)$.  Then,
  for $a\ge 0$, by \eqref{eq:Hexpl},
  \begin{equation}
    \label{eq:H_bound}
    \begin{split}
      H_h[f](a)
      &= \int_0^\infty f(x)\rho_Y
      \Big(x - \frac{a}{d} + \frac{d-1}{d}\, h\Big) \D x \\
      &= \int_{h^*}^\infty g(x)
      \rho_Y\Big(x - \frac{a + d h^* - (d-1)h}{d}\Big) \D x \\
      &= G[g]\big(a+dh^* - (d-1)h\big),
    \end{split}
  \end{equation}
  where we defined
  $G[g](a) \defeq \int_{h^*}^\infty g(x) \rho_Y(x-\frac ad) \D x$.
  Note that $G$ is strongly related to $L_{h^*}$ (see~\eqref{eq:L_h}) which suggests
  that we eventually should apply
  Proposition~\ref{prop:L_h-hypercontractivity}. We therefore write
  \begin{equation}
    \begin{split}
      \norm[\big]{ \theta_s H_h[f]}_{L^{2d}(\nu^*)}^{2d}
      &\le \int_{\R} \big(G[g](a + s + d h^* - (d-1) h)\big)^{2d}\rho_{\nu}(a + h^*)
      \D a \\
      &= \int_{\R} \big(G[g](a)\big)^{2d}\rho_{\nu}\big(a - s - (d-1)h^* +
        (d-1)h\big)
      \D a \\
      &= \int_{\R} \big(G[g](a)\big)^{2d}
      \frac{\rho_{\nu}\big(a - s - (d-1)(h^*-h)\big)}{\rho_{\nu}(a)}\rho_{\nu}(a) \D a.
    \end{split}
  \end{equation}
  Hölder's inequality with $p=5/4$ and $q=5$ then yields
  \begin{align}
    \norm[\big]{\theta_s H_h[f]}_{L^{2d}(\nu^*)}^{2d}
    &\le \norm[\big]{G[g]}_{L^{5d/2}(\nu)}^{2d}
    \norm[\Big]{\frac{\rho_{\nu}(\cdot - s -
          (d-1)(h^*-h))}{\rho_{\nu}(\cdot)}}_{L^5(\nu )}.
  \end{align}
  The function inside the second norm on the right-hand side grows at
  most exponentially at infinity and thus its $L^{5}(\nu)$-norm is
  finite, we denote it $\tilde C_{h,s}$. The first norm satisfies
  $\norm{G[g]}_{L^{5d/2}(\nu)}^{2d} \leq \norm{G[g]}_{L^{d^2+1}(\nu)}^{2d}$
  since $5d/2 \leq d^2+1$ for every $d \geq 2$. Moreover, since
  $G[f](a) = E_Y\big[1_{[h^*,\infty)}(Y+\frac{a}{d})f(Y+\frac{a}{d})\big]$,
  the hypercontractivity \eqref{eq:L_h-hypercontractivity_1} implies that
  $\norm{G[g]}_{L^{d^2 + 1}(\nu)}^{2d} \leq d\,\norm{g}_{L^{2}(\nu)}^{2d}$
  which is finite because $g\in L^2(\nu)$. This together implies
  \eqref{eq:Hboundedoperator} with with $C_{h,s} = (d \tilde C_{h,s})^{1/(2d)}$.

  (b) Let $h,h'\in I$ be such that $\abs{h-h'}<1$, and let $f\in L^2(\nu^*)$ and
  $g= \theta_{h^*}^{-1} f$. Then, analogously to \eqref{eq:H_bound}, using then
  Lemma~\ref{lem:phi_dominant},
  \begin{equation}
    \begin{split}
      &\abs[\big]{H_h[f](a)-H_{h'}[f](a)}
      \\&\le \int_{h^*}^\infty \abs{g(x)}
      \abs[\bigg]{\rho_Y\Big(x - \frac{a + d h^* - (d-1)h}{d}\Big)
        - \rho_Y\Big(x - \frac{a + d h^* - (d-1)h'}{d}\Big)} \D x
      \\&\le c \abs{h-h'} \sum_{u  =\pm 1}
      \int_{h^*}^\infty \abs{g(x)}
      \rho_Y\Big(x - \frac{a + d h^* - (d-1)h}{d} + ur\Big) \D x
      \\&\le c \abs{h-h'} \sum_{u  =\pm 1} G[\abs g]\big(a+dh^* - (d-1)h+urd\big).
    \end{split}
  \end{equation}
  Therefore, using exactly the same arguments as in the proof of (a) and the triangle
  inequality,
  \begin{equation}
    \label{eq:Hhhprimenorm}
    \norm[\big]{H_h[f]-H_{h'}[f]}_{L^{2d}(\nu^*)} \le c \abs{h-h'}
    \norm{f}_{L^2(\nu^*)},
  \end{equation}
  which proves the stated continuity.
\end{proof}

We can now compute the first partial Fréchet derivatives of the function
$F$ defined in \eqref{eq:F}.

\begin{lemma}
  \label{lem:F_f}
  The partial derivative $D_f F$ at point $(h,f)$ is given by
  \begin{equation}
    \label{eq:DfF}
    D_f F(h,f)(g) = -g + 1_{[0,\infty)} d (1- H_h[f])^{d-1} H_h[g].
  \end{equation}
  In particular, $D_f F(h, f)$ is a bounded linear operator on
  $L^2(\nu^*)$ and it depends continuously on $h\in I$ and $f\in L^2(\nu^*)$.
\end{lemma}

\begin{proof}
  We recall that for every $k\le d$ the Fréchet derivative of the power function
  $f\mapsto f^k$, viewed as a map from $L^{2d}(\nu^*)$ to
  $L^{2d/k}(\nu^*)$ is a continuous function of $f$ and is given by
  \begin{equation}
    \label{eq:derpow}
    D_f(f^k)(g) = k f^{k-1} g,
  \end{equation}
  (see, e.g., \cite[Chap.~4.3]{Zeidler}), and that the Fréchet
  derivatives satisfy the chain rule (e.g.,~Corollary to Theorem 4.D in
    \cite{Zeidler}). Therefore, using also that
  $H_h[f]\in L^{2d}(\nu^*)$ by Lemma~\ref{lem:image_of_H}(a),
  \begin{equation}
    \begin{split}
      D_f \big((1-H_h[f])^d\big)(g)
      &= - d (1- H_h[f])^{d-1} D_f H_h[f](g)
      \\&= - d (1- H_h[f])^{d-1} H_h[g],
    \end{split}
  \end{equation}
  where in the last step we used the fact that $H_h$ is a linear operator
  and thus $D_f H_h[f](g) = H_h[g]$. Recalling the definition
  \eqref{eq:F} of $F$, formula \eqref{eq:DfF} directly follows. The fact
  that $D_f F(h,f)$ is a bounded linear operator on $L^2(\nu^*)$ then
  follows by Lemma~\ref{lem:image_of_H}(a) and Hölder's inequality.

  To prove the continuity of $D_f F(h,f)$, let $h,h'\in I$ and
  $f,f',g\in L^2(\nu^*)$. Ignoring the non-essential prefactor
  $d1_{[0,\infty)}$,
  $D_fF(h,f)(g)- D_fF(h',f')(g)$, can be written as
  \begin{equation}
    \begin{split}
      &(1- H_h[f])^{d-1} H_h[g]- (1- H_{h'}[f'])^{d-1} H_{h'}[g]
      \\&=\big((1- H_h[f])^{d-1}- (1- H_{h'}[f'])^{d-1}\big) H_h[g]
      +  (1- H_{h'}[f'])^{d-1} \big(H_h[g]-H_{h'}[g]\big).
    \end{split}
  \end{equation}
  By Hölder's inequality and \eqref{eq:Hhhprimenorm}, the $L^2(\nu^*)$-norm of
  the second summand is bounded by
  $C(1+\norm{f'}_{L^2(\nu^*)})^{d-1}\norm{g}_{L^2(\nu^*)}\abs{h-h'}$. The first
  summand can be rewritten using the formula
  $a^{k}-b^k=(a-b)\sum_{i=0}^{k-1}a^{i}b^{k-1-i}$, and the so arising term
  $H_h[f]-H_{h'}[f']$ can be expanded as
  $(H_h[f]-H_{h'}[f])+(H_{h'}[f]-H_{h'}[f'])$. Therefore, using the linearity of
  $H_h$, Lemma~\ref{lem:image_of_H}(a,b), and Hölder's inequality again, the
  $L^2(\nu^*)$-norm of the first summand is bounded by
  \begin{equation}
    C\norm{g} \big[
      \abs{h-h'}\big(1+\norm{f}+\norm{f'}\big)^{d-1}
      + \norm{f-f'}\big(1+\norm{f}+\norm{f'}\big)^{d-2}
      \big],
  \end{equation}
  where all norms are in $L^2(\nu^*)$. The continuity of
  $(f,h)\mapsto D_f(f,h)$ then directly follows from these estimates.
\end{proof}

\begin{lemma}
  \label{lem:F_h}
  The partial derivative $D_h F$ at point $(h,f)$ is given by
  \begin{equation}
    \label{eq:DhF}
    D_h F(h, f) =  1_{[0,\infty)}d\big(1-H_h[f]\big)^{d-1}H_h'[f] ,
  \end{equation}
  where, for $h \in I$ and $f \in L^2(\nu^*)$,
  $H'_h[f]:\mathbb R\to  L^{2d}(\nu^*)$ is given by
  \begin{equation}
    \label{eq:Hprimedef}
    H_h'[f](a) = \frac{d-1}d \int_0^\infty f(x) \rho'_Y\Big(x - \frac{a}{d} +
      \frac {d-1}d h\Big) \D x,
  \end{equation}
  with $\rho'_Y$ being the derivative of $\rho_Y$. In particular,
  $D_h F(h,f)\in L^2(\nu^*)$ and it is a continuous
  function of $h\in I$ and $f\in L^2(\nu^*)$.
\end{lemma}

\begin{proof}
  We start by showing that for any $f\in L^2(\nu^*)$ the $h$-derivative of
  $H_h[f]$ is given by $D_h H_h[f] = H_h'[f] \in L^{2d}(\nu^*)$. To see this we
  have to check that
  \begin{equation}
    \label{eq:Hprime}
    \lim_{s\to0}\frac{1}{s} \big(H_{h+s}[f] - H_h[f]\big)
    = H'_h[f]\quad \text{ in } L^{2d}(\nu^*).
  \end{equation}
  We first show the pointwise convergence.
  By \eqref{eq:Hexpl}, for fixed $a \in \R$, it holds
  \begin{equation}
    \begin{split}
      &\frac{1}{s}\big(H_{h+s}[f](a) - H_h[f](a)\big) \\
      &\qquad = \int_0^\infty f(x) \frac 1s
      \bigg(\rho_Y\Big(x - \frac ad + \frac{d-1}d (h+s)\Big) -
        \rho_Y\Big(x-\frac ad +\frac {d-1}d h\Big)\bigg) \D x \\
      &\qquad \eqdef \int_0^\infty f(x) \Phi_s\Big(x - \frac ad  + \frac{d-1}d
        h\Big) \D x.
    \end{split}
  \end{equation}
  Obviously,
  $\lim_{s\to 0}\Phi_s(x - \frac ad  + \frac{d-1}{d}h)= \frac{d-1}d \rho'_Y(x - \frac ad + \frac{d-1}d h)$,
  and, by Lemma~\ref{lem:phi_dominant},
  $\max_{|s|<1}\abs{\Phi_s(y)}\le c(\rho_Y(y+r) + \rho_Y(y-r)) \eqdef \bar{\Phi}(y)$.
  Since $f \in L^2(\nu^*) \subset L^1(\nu^*)$, it holds that
  $f\rho_{\nu^*}\in L^1(\D x)$. Moreover, since the variance of $\nu^*$ is
  larger than the variance of $Y$, that is $\sigma_\nu^2 > \sigma_Y^2$,
  the ratio $\rho_Y(y+c)/\rho_{\nu^*}(y)$ is bounded for any
  $c\in \mathbb R$. As consequence,
  $f(\cdot)\bar{\Phi}(\cdot - \frac ad + \frac{d-1}d h)\in
  L^1([0,\infty),\D x) $, and thus, by the dominated convergence theorem,
  \begin{equation}
    \begin{split}
      &\lim_{s\to 0}\frac{1}{s}\big(H_{h+s}[f](a) - H_h[f](a)\big)
      = \int_0^\infty f(x) \lim_{s \to \infty} \Phi_s\Big(x-\frac ad
        +\frac{d-1}{d}h\Big) \D x \\
      &= \frac{d-1}d \int_0^\infty f(x) \rho'_Y\Big(x - \frac{a}{d} + \frac{d-1}d
        h\Big) \D x
      = H_h'[f](a),
    \end{split}
  \end{equation}
  which establishes the pointwise convergence in \eqref{eq:Hprime}.

  To show the convergence in $L^{2d}(\nu^*)$ we observe, by
  Lemma~\ref{lem:phi_dominant} again, that the function
  \begin{equation}
    \bar{H}_h[f](a) \defeq
    \int_0^\infty \abs{f(x)}
    \bar{\Phi}\Big(x - \frac ad + \frac{d-1}d h\Big)\D x
    = c\big(H_h[f](a+r) + H_h[f](a-r)\big)
  \end{equation}
  dominates $\abs[\big]{\frac{1}{s}(H(h+s,f)-H(h,s))}$  for all small
  $\abs{s}$. Moreover, by
  Lemma~\ref{lem:image_of_H},
  $\theta_{\pm r}{H}_h{[f] \in L^{2d}}(\nu^*)$ and thus also
  $\bar H_h[f]\in L^{2d}(\nu^{*})$. The $L^{2d}(\nu^*)$ convergence in
  \eqref{eq:Hprime} thus follows by another application of the dominated
  convergence theorem.

  Claim \eqref{eq:DhF} then follows from \eqref{eq:Hprime} and the
  definition \eqref{eq:F} of $F$ by
  the chain rule:
  \begin{align}
    D_h F(h,f) &= - D_h\Big( 1_{[0, \infty)}   \big( 1 -
          H_h[f]\big)^d\Big)
    =  1_{[0,\infty)}d\big(1-H_h[f]\big)^{d-1} H'_h[f],
  \end{align}
  as required.

  Finally, we show that $(h,f)\mapsto D_h F(h,f)$ is continuous. We first
  observe that by similar arguments as in the proof of
  Lemma~\ref{lem:image_of_H}, one can show that $H'_h[f]$ satisfies
  analogous estimates as $H_h[f]$, namely, for $h,h' \in \R$ and
  $f \in L^2(\nu^*)$,
  \begin{equation}
    \begin{split}
      \label{eq:Hprimenorms}
      \norm{H'_h[f]}_{L^{2d}(\nu^*)} &\leq C_h \norm{f}_{L^2(\nu^*)},
      \\
      \norm{H'_h[f]- H'_{h'}[f]}_{L^{2d}(\nu^*)}
      &\leq c \abs{h-h'} \norm{f}_{L^{2}(\nu^*)}.
    \end{split}
  \end{equation}
  Then, again similarly to the proof of the continuity of $D_f F(h,f)$,
  ignoring the non-essential prefactor $d 1_{[0,\infty)}$,
  $D_h F(h,f) - D_h F(h',f')$ can be written as
  \begin{equation}
    \begin{split}
      &(1- H_h[f])^{d-1} H'_h[f]- (1- H_{h'}[f'])^{d-1} H'_{h'}[f']
      \\&=\big((1- H_h[f])^{d-1}- (1- H_{h'}[f'])^{d-1}\big) H'_h[f]
      +  (1- H_{h'}[f'])^{d-1} \big(H'_h[f]-H'_{h'}[f']\big).
    \end{split}
  \end{equation}
  From this the continuity of $D_h F(h,f)$ follows by the same arguments
  as before, replacing some of the estimates on $H_h[f]$ by analogous
  estimates \eqref{eq:Hprimenorms} when needed.
\end{proof}

\begin{lemma}
  \label{lem:second_part_derv_F}
  The relevant second partial Fréchet derivatives of $F$ are continuous
  and given by
  \begin{equation}
    \label{eq:F_ff}
    D_{ff}F(h,f) (g_1, g_2) = - 1_{[0,\infty)} d (d-1)
    \big(1-H_h[f]\big)^{d-2} H_h[g_1] H_h[g_2].
  \end{equation}
  and
  \begin{equation}
    \label{eq:F_hf}
    \begin{split}
      D_{hf}F(h,f) (g) = &- 1_{[0,\infty)} d (d-1)
      \big(1-H_h[f]\big)^{d-2} H_h[g] H_h'[f]
      \\ &+  1_{[0,\infty)} d \big(1-H_h[f]\big)^{d-1} H_h'[g].
    \end{split}
  \end{equation}
\end{lemma}

\begin{proof}
  For fixed $h\in\R$, $D_f F(h, \cdot)$ can be written as a composition of functions
  $D_f F(h, \cdot) = F_2 \circ F_1$ with $F_1: L^2(\nu^*) \to L^{2d/(d-1)}(\nu^*)$,
  $F_1(f) = 1_{[0,\infty)} d (1-H_h[f])^{d-1}$ and
  $F_2: L^{2d/(d-1)}(\nu^*) \to B(L^2(\nu^*), L^2(\nu^*)), F_2(l)(g) = -g + l \cdot H_h[g]$.
  By \eqref{eq:derpow}, $F_1$ is $C^1$ with
  $DF_1(f)(g) = -1_{[0,\infty)} d (d-1) (1-H_h[f])^{d-2}H_h[g]$. Further,
  via calculating the term $F_2(l+u) - F_2(l)$, one gets
  $DF_2(l)(u)(g) = u \cdot H_h[g]$. Using the chain rule
  gives $D_{ff}F(h,f)$ as stated in \eqref{eq:F_ff}.

  To compute $D_{hf}F(h,f)$, we fix $h\in \R$ and write $D_h F(h, \cdot)$ as a
  multiplication of functions $D_h F(h, \cdot) = F_1(\cdot) H'_h[\cdot]$, where
  $F_1$ is as described above. Noting that due to linearity, $D_f H'_h[f](g) = H'[g]$,
  and using the product rule for Fréchet derivatives (see,
  e.g.,~Standard Example 3 to Theorem 4.D in \cite{Zeidler}) gives
  \eqref{eq:F_hf}.

  To see the continuity of $D_{ff} F(h,f)$, let $h,h' \in I$ and
  $f,f',g \in L^2(\nu^*)$. The difference
  $D_{ff} F(h,f)(g_1, g_2) - D_{ff} F(h',f')(g_1, g_2)$ can be written as
  (again ignoring the prefactor $d 1_{[0,\infty)}$)
  \begin{equation}
    \begin{split}
      (1- &H_h[f])^{d-2} H_h[g_1] H_h[g_2]
      - (1- H_{h'}[f'])^{d-2} H_{h'}[g_1] H_{h'}[g_2]
      \\&=\big((1- H_h[f])^{d-2}- (1- H_{h'}[f'])^{d-1}\big) H_h[g_1] H_h[g_2]
      \\&\quad+ \frac{1}{2} (1- H_{h'}[f'])^{d-2} \big(H_h[g_1]-H_{h'}[g_1]\big)\big(H_h[g_2]+H_{h'}[g_2]\big)
      \\&\quad+ \frac{1}{2} (1- H_{h'}[f'])^{d-2} \big(H_h[g_1]+H_{h'}[g_1]\big)\big(H_h[g_2]-H_{h'}[g_2]\big).
    \end{split}
  \end{equation}
  Using the Hölder's inequality for three functions on
  every summand together with similar arguments as in the proofs of the
  Lemmas~\ref{lem:F_f} and \ref{lem:F_h}, then implies the continuity of
  $(f, h) \mapsto D_{ff}F(f,h)$. Similarly, for $D_{fh}F(h,f)$, the
  analogous decomposition of $D_{fh}F(h,f) - D_{fh}F(h',f')$ together
  with previous arguments gives the continuity.
\end{proof}

It is left to check that properties (c) and (d) of
Proposition~\ref{prop:bifurcation1} are satisfied.

\begin{lemma}
  \label{lem:dimN_codimR}
  Let $\chi^* = \theta_{h^*}\chi_{h^*}$ be a shift of $\chi_{h^*}$.
  It holds that
  \begin{equation}
    N(D_f F(h^*, 0)) = R(D_f F(h^*,0))^\perp = \spn\set{\chi^*},
  \end{equation}
  where $^\perp$ stands for the orthogonal complement in $L^2(\nu^*)$.
  In particular,
  \begin{equation}
    \dim N(D_f F(h^*, 0)) = \codim R(D_f F(h^*,0)) = 1.
  \end{equation}
\end{lemma}

\begin{proof}
  By Lemma~\ref{lem:F_f} and \eqref{eq:HLrelation} it holds
  \begin{equation}
    \label{eq:Dhstar}
    D_f F(h^*, 0)(g) = -g + 1_{[0,\infty)} d H_{h^*}[g] = -g + \theta_{h^*}
    L_h[\theta^{-1}_{h^*} g].
  \end{equation}
  Therefore $g\in N(D_f F(h^*,0))$ if and only if $\theta_{h^*}^{-1}g$ is an
  eigenfunction of $L_{h^*}$ corresponding to eigenvalue $1$. By
  Proposition~\ref{prop:L_h-chi-connection}, $\chi$ is the only such
  eigenfunction, and thus
  \begin{equation}
    N(D_f F(h^*, 0)) = \spn\set{\theta_{h^*} \chi_{h^*}} =
    \spn\set{\chi^*}
  \end{equation}
  and its dimension is equal to one.

  The property that $l \in R(D_f F(h^*, 0))^\perp$ is equivalent to
  $\inp[\big]{l}{D_f F(h^*, 0)(g)}_{\nu^*} = 0$
  for all $g \in L^2(\nu^*)$. However, since by
  Proposition~\ref{prop:L_h-chi-connection} $L_h$ is self-adjoint on
  $L^2(\nu)$,  \eqref{eq:Dhstar} implies that $D_f F(h^*,0)$ is
  self-adjoint on  $L^2(\nu^*)$. Therefore, this is equivalent to
  $0=\inp[\big]{l}{D_f F(h^*, 0)(g)}_{\nu^*} = \inp[\big]{D_f F(h^*,0)(l)}{g}_{\nu^*} $
  for all $g \in L^2(\nu^*)$. However, this is true iff
  $l \in N(D_f F(h^*,0)) = \spn\set{\chi^*}$.
\end{proof}

\begin{lemma}
  \label{lem:bifurcation_last_req}
  It holds
  \begin{equation}
    D_{hf} F(h^*, 0)(\chi^*) \not\in R(D_f F(h^*, 0)).
  \end{equation}
\end{lemma}

\begin{proof}
  By Lemma~\ref{lem:dimN_codimR}, $g \in R(D_f F(h^*, 0))$ iff $g$ orthogonal to
  $ \spn\set{\chi^*}$. We thus only need to show that
  $\inp[\big]{\chi^*}{D_{hf}F(h^*, 0)(\chi^*)}_{\nu^*} \neq 0$.

  Recall that $\sigma_Y^2 = \frac{d+1}{d}$, and thus
  $\rho'_Y(x) = -\frac{d}{d+1} x \rho_Y(x)$.  By
  Lemma~\ref{lem:second_part_derv_F}, for $a,s \in \R$, using that
  $H_{h^*}[0]= 0$ and the definition \eqref{eq:Hprimedef} of $H_h'$,
  \begin{equation}
    \begin{split}
      &D_{hf}F(h^*, 0)(\chi^*)(a)
      = 1_{[0,\infty)}(a) d H'_{h^*}[\chi^*](a)
      \\ &= 1_{[0,\infty)}(a) (d-1)  \int_0^\infty \chi^*(x)
      \rho'_Y\Big(x- \frac ad + \frac{d-1}d h^*\Big) \D x
      \\&= - 1_{[0,\infty)}(a) \frac{d(d-1)}{d+1}
      \int_0^\infty \chi^*(x) \Big(x - \frac ad + \frac{d-1}d h^*\Big)
      \rho_Y\Big(x- \frac ad + \frac{d-1}{d} h^*\Big) \D x.
    \end{split}
  \end{equation}
  Writing $x-\frac ad + \frac {d-1}d h^* = x + h^* - \frac 1d (a+h^*)$,
  and observing that $(\cdot +h^*) = \theta_{h^*} \Id$ with $\Id$ being the
  identity map on $\mathbb R$,
  this can be written as
  \begin{equation}
    \begin{split}
      &= 1_{[0,\infty)}(a) \frac{d-1}{d+1}
      \Big((a+h^*)  H_{h^*}[\chi^*](a)
        - d H_{h^*}[\chi^*\theta_{h^*} \Id ](a)\Big)
      \\&= 1_{[0,\infty)}(a) \frac{d-1}{d+1}
      \Big(\frac 1 d (\chi^*\theta_{h^*}\Id)(a)
        - d H_{h^*}[\chi^*\theta_{h^*} \Id ](a)\Big),
    \end{split}
  \end{equation}
  where in the last equality we used that $\chi^*$ is an eigenfunction of
  $H_{h^*}$ with eigenvalue $d^{-1}$, by \eqref{eq:HLrelation}.
  Therefore, using $\chi^* = 1_{[0,\infty)}\chi^*$ and the
  self-adjointness of $H_{h^*}$ on $L^2(\nu^2)$,
  \begin{equation}
    \label{eq:inp_chi_Fhf}
    \begin{split}
      &\inp[\big]{\chi^*}{D_{hf}F(h^*, 0)(\chi^*)}_{\nu^*}
      \\ &\qquad=
      \frac{d-1}{d+1} \Big(
        \frac{1}{d}\inp[\big]{\chi^*}{\chi^*\theta_{h^*}\Id}_{\nu^*}
        - d \inp[\big]{\chi^*}{H_{h^*}[\chi^*\theta_{h^*}\Id]}_{\nu^*}
        \Big)
      \\ &\qquad=
      \frac{d-1}{d+1} \Big(
        \frac{1}{d}\inp[\big]{\chi^*}{\chi^*\theta_{h^*}\Id}_{\nu^*}
        - d \inp[\big]{H_{h^*}[\chi^*]}{\chi^*\theta_{h^*}\Id}_{\nu^*}
        \Big)
      \\ &\qquad=
      \frac{d-1}{d+1} \Big(
        \frac{1}{d}\inp[\big]{\chi^*}{\chi^*\theta_{h^*}\Id}_{\nu^*}
        -  \inp[\big]{\chi^*}{\chi^*\theta_{h^*}\Id}_{\nu^*}
        \Big)
      \\&\qquad=
      -\frac{(d-1)^2}{d(d+1)} \inp{\chi}{\chi \Id}_{\nu} \neq 0,
    \end{split}
  \end{equation}
  where in the last step we first applied $\theta^{-1}_{h^*}$ and then used
  the fact that $\chi(a) > 0$ iff $a \in [h^*,\infty)$, and thus $\inp{\chi}{  \chi \Id}_{\nu} > 0$.
\end{proof}

We can now prove Theorem~\ref{thm:percolation_prob}.

\begin{proof}[Proof of Theorem \ref{thm:percolation_prob}]
  As already explained, we apply Proposition \ref{prop:bifurcation1} to
  the function $F$ defined in \eqref{eq:F}, with $X=Y=L^2(\nu^*)$, $h$ and $f$ playing
  the role of
  $t$ and $x$, respectively, and with $h^*$ corresponding to $t_0$. By
  Lemmas~\ref{lem:F_f}--\ref{lem:bifurcation_last_req}, the requirements
  (a)--(d) of this proposition are satisfied with $x_0=\chi^*$, and $D_{ff}F$ is continuous.
  Taking for $Z=(\chi^*)^\perp$ for sake of concreteness, there is thus a neighbourhood
  $U\subset \mathbb R\times L^2(\nu^*)$ of $(h^*,0)$ such that all
  non-trivial (that is non-zero) solutions of $F(h,f) = 0$ in $U$ can be
  written as
  $\set{(\varphi (\alpha), \alpha \chi^* + \alpha \psi(\alpha)) : \abs{\alpha} < a}$
  for some $a > 0$ and $C^1$ functions $\varphi : (-a,a) \to \R$,
  $\psi: (-a, a) \to (\chi^*)^\perp \subset L^2(\nu^*)$ with $\varphi (0) = h^*$, $\psi(0) = 0$.
  Further, by \eqref{eq:bifurcation2},
  \begin{equation}
    \label{eq:derivatives}
    \frac{1}{2} D_{ff}F(h^*, 0)(\chi^*, \chi^*) + D_f F(h^*, 0)(\psi'(0))
    + \varphi '(0)D_{hf}F(h^*, 0)(\chi^*) = 0.
  \end{equation}

  Since, $\varphi (\alpha) = h^* + \alpha \varphi '(0) + o(\alpha )$ as
  $\alpha \to 0$, and thus
  $\varphi^{-1}(h) = \frac{h-h^*}{\varphi '(0)} + o(h-h^*)$ as $h\to h^*$,
  the non-trivial solution $\alpha \chi^* + \alpha \psi(\alpha )$ can be expanded as
  a function of $h$, in a neighbourhood of $h^*$,
  \begin{equation}
    \label{eqn:fexpansion}
     \varphi^{-1}(h) \chi^* + \varphi^{-1}(h) \psi(\varphi^{-1}(h))
    = \frac{h - h^*}{\varphi '(0)} (\chi^* + r^*_h),
  \end{equation}
  where the reminder function $r_h^*$ satisfies
  $\lim_{h\to h^*} \norm{r^*_h}_{L^2(\nu^*)} =0 $.

  Recall now from \eqref{eq:shifted_percolation_eq}, \eqref{eq:F} that
  $F(h,f)=0$ is the equation for the shifted forward percolation
  probability $\widetilde \eta_h $. Shifting everything back by
  $\theta_{h}^{-1}$, using the fact that by the continuity of the percolation
  probabilities (Corollary~\ref{cor:continuity}), the solution obtained
  from the bifurcation analysis must agree with $\eta^{+}(h,a)$, we
  obtain from \eqref{eqn:fexpansion} that
  \begin{equation}
    \label{eq:eta_tilde}
    \eta^+(h,\cdot) = \frac{h - h^*}{\varphi '(0)}
    \theta_{h}^{-1}(\chi^* + r^*_h)
    = \frac{h - h^*}{\varphi '(0)}
    \big(\chi  + (\theta_{h^*-h}\chi -\chi ) +
      \theta_{-h}r_h^*\big).
  \end{equation}

  We first show that
  \begin{equation}
    \label{eq:rhnorm}
    r_h^+ := (\theta_{h-h^*}\chi -\chi ) +\theta_{-h}r_h^* \to 0
    \qquad\text{as $h\to h^*$ in }L^{2-\varepsilon }(\nu ).
  \end{equation}
  For the first summand $\theta_{h-h^*}\chi -\chi$, this follows easily
  from the continuity of
  $\chi $ (Proposition~\ref{prop:L_h-chi-connection}), growth estimates on
  $\chi $ (Proposition~\ref{prop:chi_bounds}) and the dominated convergence
  theorem. For the second summand, it holds that
  \begin{equation}
    \begin{split}
      \norm{\theta_{-h}r_h^*}_{L^{2-\varepsilon }(\nu )}^{2-\varepsilon}
      &= \int \abs{r_h^*(a-h)}^{2-\varepsilon }
      \rho_\nu (a) \D a
      \\&= \int \abs{r_h^*(a)}^{2-\varepsilon }
      \frac{\rho_\nu (a+h)}{\rho_{\nu ^*}(a)} \rho_{\nu^*}(a)  \D a
      \\&\le  \norm{r_h^*}_{L^2(\nu^*)}^{2-\varepsilon }
      \norm[\Big]{\frac{\rho_\nu (a+h)}{\rho_{\nu ^*}(a)}}_{L^{2/\varepsilon
      }(\nu^*)},
    \end{split}
  \end{equation}
  where we applied Hölder's inequality on the last line.
  The first factor on the right-hand side converges to $0$ as $h\to h^*$,
  and the second factor remains bounded, which completes the proof of
  \eqref{eq:rhnorm}.

  We proceed by computing $\varphi '(0)$. To this end we project
  \eqref{eq:derivatives} to the $\chi^*$ direction by applying
  $\inp{\chi^*}{\cdot}_{\nu^*}$ on both sides. Note that, by
  Lemma~\ref{lem:dimN_codimR}, the range of $D_f F(h^*,0)$ is orthogonal
  to $\chi^*$, and thus
  $\inp{\chi^*}{D_f F(h^*, 0)(\psi'(0))}_{\nu^*} = 0$. Hence,
  \begin{equation}
    \varphi '(0) = - \frac{\inp{\chi^*}{D_{ff}F(h^*,0)(\chi^*, \chi^*)}_{\nu^*}}
    {2 \inp{\chi^*}{D_{hf}F(h^*,0)(\chi^*)}_{\nu^*}}.
  \end{equation}
  The scalar product in the denominator was computed in
  \eqref{eq:inp_chi_Fhf}.  For the numerator, by
  Lemma~\ref{lem:second_part_derv_F},
  \begin{equation}
    D_{ff}F(h^*, 0)(\chi^*, \chi^*)
    = -d (d-1) H_{h^*}[\chi^*]H_{h^*}[\chi^*]
    = -\frac{d-1}{d} (\chi^*)^2,
  \end{equation}
  and therefore,
  \begin{equation}
    \inp{\chi^*}{D_{ff}F(h^*,0)(\chi^*, \chi^*)}_{\nu^*}
    = -\frac{d-1}{d} \inp{\chi^*}{(\chi^*)^2}_{\nu^*}
    = -\frac{d-1}{d} \inp{\chi}{\chi^2}_{\nu}.
  \end{equation}
  As consequence,
  \begin{equation}
    \label{eq:h_prime}
    \varphi '(0) = - \frac{1}{2} \frac{d+1}{d-1}
    \frac{\inp{\chi}{\chi^2}_\nu}{\inp{\chi^2}{ \Id}_\nu}.
  \end{equation}
  Claim~\eqref{eq:thm_percolation_prob_1} of the theorem then follows directly
  from \eqref{eq:eta_tilde}, \eqref{eq:rhnorm} and \eqref{eq:h_prime}.

  To obtain the asymptotics \eqref{eq:thm_percolation_prob_2} of $\eta$,
  note that for any $a \geq h^*$, using similar arguments in the proof of
  the recursion property \eqref{eq:self-reference},
  \begin{equation}
    \begin{split}
      \eta(h,a) &=
      P_a[\abs{\mathcal{C}_o^h} = \infty]
      = 1 - P_a[\abs{\mathcal{C}_o^h} < \infty] \\
      &= 1 - \Big( E_Y\Big[1-\eta^+\Big(h, \frac ad +Y\Big)\Big] \Big)^{d+1} \\
      &= (d+1) E_Y\Big[\eta^+\Big(h, \frac ad + Y\Big)\Big] \big(1+o(1)\big)
    \end{split}
  \end{equation}
  as $h \uparrow h^*$. The same argument applied to $\eta^{+}$ implies
  \begin{equation}
    \eta^*(h,a) = d E_Y\Big[\eta^+\Big(h, \frac ad +
        Y\Big)\Big]\big(1+o(1)\big)
  \end{equation}
  and thus, for every $a\in \mathbb R$,
  \begin{equation}
    \lim_{h\uparrow h^*}\frac{\eta(h, a)}{\eta^+(h,a)}= \frac{d+1}{d}.
  \end{equation}
  which together with \eqref{eq:thm_percolation_prob_1} shows
  \eqref{eq:thm_percolation_prob_2} and completes the proof.
\end{proof}

\def\arxiv#1{Preprint, available at \href{http://arxiv.org/abs/#1}{arXiv:#1}}
\providecommand{\bysame}{\leavevmode\hbox to3em{\hrulefill}\thinspace}
\providecommand\MR{}
\renewcommand\MR[1]{\relax\ifhmode\unskip\spacefactor3000
\space\fi \MRhref{#1}{#1}}
\providecommand\MRhref{}
\renewcommand{\MRhref}[2]%
{\href{http://www.ams.org/mathscinet-getitem?mr=#1}{MR#2}}
\providecommand{\href}[2]{#2}

\end{document}